\documentclass[12pt,letterpaper]{amsart}

\usepackage{euscript,amsfonts,amssymb,amsmath,amscd,amsthm,enumerate,hyperref}

\usepackage{tikz,bm,float}
\usetikzlibrary{calc}
\usetikzlibrary{patterns}
\colorlet{lgray}{white!85!black}
\colorlet{lred}{white!75!red}

\usepackage{comment}

\usepackage{graphicx}

\usepackage{color}

\usepackage[margin=0.9in]{geometry}
\newtheorem{theorem}{Theorem} %[section]
\newtheorem*{theorem*}{Theorem}
\newtheorem{lemma}[theorem]{Lemma}
\newtheorem{definition}[theorem]{Definition}
\newtheorem{proposition}[theorem]{Proposition}

\theoremstyle{remark}
\newtheorem{remark}[theorem]{Remark}

\numberwithin{equation}{section} \numberwithin{theorem}{section}

\newcommand{\la}{\lambda}

\usepackage{MnSymbol}

\usepackage{skak}

\usepackage{adjustbox}
\usetikzlibrary{arrows}
\usetikzlibrary{decorations.pathreplacing}

\newcommand{\prech}{\prec_{\mathsf{h}}}

\newcommand{\precv}{\prec_{\mathsf{v}}}

\newcommand{\precd}{\prec_{\mathsf{\leq d}}}
\newcommand{\nprecd}{\prec_{\mathsf{> d}}}

\title[Macdonald-positive specializations of the algebra of symmetric functions]
{Macdonald-positive specializations of the algebra of symmetric functions: Proof of the Kerov conjecture}

\author{Konstantin Matveev}

\address[Konstantin Matveev]{ Department of Mathematics, Brandeis University, Waltham, MA, USA. E-mail: kosmatveev@gmail.com}

\begin{document}

\begin{abstract}
We prove the classification of homomorphisms from the algebra of symmetric functions to $\mathbb{R}$ with non-negative values on Macdonald symmetric functions $P_{\lambda}$, that was  conjectured by S.V. Kerov in 1992. 
\end{abstract}

\maketitle

\tableofcontents

%\begin{align}
%Q_{(m, j)} = Q_{(m)}Q_{(j)} + \sum_{i=0}^{j-1}\frac{t^{j-i}\left(t^{-1}; q\right)_{j-i}}{\left(q; q\right)_{j-i}} \cdot \frac{(1-q^{m+j-2i})\left(q^{m-i-1}; q\right)_{j-i-1}}{\left(tq^{m-i}; q\right)_{j-i}}Q_{(m+j-i)}Q_{(i)}
%\end{align}

\section{Introduction}
\subsection{Edrei-Thoma theorem and Kerov conjecture}
In this section we recall the Kerov conjecture and briefly review the history of its special case, the Edrei-Thoma theorem.
A (one-sided) {\it P\'{o}lya frequency sequence} is a sequence $\{a_{n}\}_{n=1}^{\infty}$ of real numbers, such that the infinite upper unitriangular matrix \footnote{Also called a Toeplitz matrix.}
\begin{align*}
A = \begin{bmatrix}
    1 & a_{1} & a_{2} & a_{3} & a_{4} & \dots \\
    0 & 1 & a_{1} & a_{2} & a_{3} & \dots \\
    0 & 0 & 1 & a_{1} & a_{2} & \dots \\
    0 & 0 & 0 & 1 & a_{1} & \dots \\
    \vdots & \vdots & \vdots & \vdots & \ddots & \vdots 
\end{bmatrix} 
\end{align*}
is {\it totally positive}, i.e. all finite minors of $A$ are non-negative. Both terms were coined by I.J.~Schoenberg, who in late 1940s and early 1950s has worked on P\'{o}lya frequency sequences, functions, and kernels. His motivation came from questions in analysis, namely bounding the number of real roots of a polynomial in a finite interval, studying variation-diminishing transformations, and approximation by analytic functions. See \cite{Sch88} for Schoenberg's own account of his life and work, see \cite{Kar}, \cite{Pin} for in-depth review of the subject of totally positive matrices. In \cite[p.~367]{Sch48} Schoenberg made the following conjecture.
\begin{proposition}
\label{Edrei-Thoma}
$\{a_{n}\}_{n=1}^{\infty}$ is a P\'{o}lya frequency sequence if and only if
\begin{align}
\label{ETgenfun}
\displaystyle 1+\sum_{n=1}^{\infty} a_{n}z^{n} = e^{\gamma z} \frac{\prod_{j=1}^{\infty}(1+\beta_{j}z)}{\prod_{i=1}^{\infty}(1-\alpha_{i}z)} 
\end{align}
for some $\alpha_{i}, \beta_{j}, \gamma \geq 0$, such that $\sum_{i =1}^{\infty} \alpha_{i} + \sum_{j =1}^{\infty} \beta_{j} < \infty$.
\end{proposition}
Showing that $\{a_{n}\}_{n=1}^{\infty}$ defined by the generating function \eqref{ETgenfun} is indeed a P\'{o}lya frequency sequence is relatively straightforward. The hard part is to show that there are no other P\'{o}lya frequency sequences.
This conjecture was proved in 1952 in a series of three papers \cite{Whit}, \cite{ASW}, \cite{Edr}. The proof naturally splits into two parts that are using different methods. First, in \cite{Whit}, \cite{ASW} it was shown that the statement can be reduced to 
\begin{proposition}
\label{ETreduced}
If $\{a_{n}\}_{n=1}^{\infty}$ is a P\'{o}lya frequency sequence, such that $1+\sum_{n=1}^{\infty} a_{n}z^{n}$ gives an entire function with no zeroes, then $1+\sum_{n=1}^{\infty} a_{n}z^{n} = e^{\gamma z}$ for some $\gamma \geq 0$.
\end{proposition}
This reduction was proved by  showing that if 
the generating function of a P\'{o}lya frequency sequence has the smallest pole at $1/\alpha$, then multiplying it by $(1-\alpha z)$ produces a generating function of another P\'{o}lya frequency sequence \footnote{It is not a priori clear that the generating function is meromorphic. Proving this  was part of the argument.}. Thus, by consequently applying such multiplications, one can remove all poles of \eqref{ETgenfun}, and then similarly remove all zeroes. The  second part of the proof, namely proof of the Proposition \ref{ETreduced}, was done in \cite{Edr} via an application of the Nevanlinna theory of meromorphic functions. This complex analytic machinery gives tools to describe the  distribution of solutions of the equation $f(z)=a$. A.~Edrei was able to prove Proposition \ref{ETreduced} by applying the Nevanlinna theory to function \begin{align*}f(z) = \frac{1+\sum_{n=1}^{\infty} a_{2n}z^{2n}}{\left(1+\sum_{n=1}^{\infty} a_{n}z^{n}\right)^{2}}.
\end{align*}
Independently (and priorly) partial results in the direction of Proposition \ref{Edrei-Thoma} were obtained by F.R.~Gantmacher and M.G.~Krein in connection with boundary value problems arising in vibration problems, see \cite{GK}.  

Independently, Proposition \ref{Edrei-Thoma}  was discovered by E.~Thoma \cite{Th} in the context of classifying normalized characters of the infinite symmetric group $S_{\infty}$, i.e. the group of finitary permutations of the countable set. His proof was very similar to that of \cite{Whit}, \cite{ASW}, \cite{Edr}. 

Another, completely new proof of Proposition \ref{Edrei-Thoma} was given in \cite{VK81} based on the paradigm of the asymptotic representation theory that was discovered by A.M.~Vershik and S.V.~Kerov. Since then the asymptotic representation theory has experienced many interesting developments, for a review see \cite{V03}, \cite{BO}. They have also discovered that it is instructive to restate  Proposition \ref{Edrei-Thoma} in the language of symmetric functions. See \cite{Mac} for an in-depth review of the theory of symmetric functions and \cite{Ful} for connections to representation theory and combinatorics.  

Denote by $\Lambda$ the algebra of symmetric power series of bounded degree (called {\it symmetric functions}) in countably many variables $x_{1}, x_{2}, x_{3}, \ldots$ over $\mathbb{R}$. Let
\begin{align}
h_{r} := \sum_{1 \leq i_{1} \leq i_{2} \leq \cdots \leq i_{r}} x_{i_{1}}x_{i_{2}} \cdots x_{i_{r}}, \qquad e_{r} := \sum_{1 \leq i_{1} < i_{2} < \cdots < i_{r}} x_{i_{1}}x_{i_{2}} \cdots x_{i_{r}}, \qquad  \quad p_{r} := \sum_{i \geq 1} x_{i}^{r}
\end{align}
be the $r$-th {\it complete symmetric function}, the $r$-th {\it elementary symmetric function}, and the $r$-th {\it power sum}, respectively. We also set $h_{0} =e_{0}=p_{0}:=1$ and $h_{r} =e_{r} = p_{r} :=0$ for $r<0$. Then one can show that
\begin{align}
\Lambda = \mathbb{R}[h_{1}, h_{2}, h_{3}, \ldots] = \mathbb{R}[e_{1}, e_{2}, e_{3}, \ldots] = \mathbb{R}[p_{1}, p_{2}, p_{3}, \ldots].
\end{align}
Moreover, $\Lambda$ is freely generated as an $\mathbb{R}$-algebra by either $\{h_{n}\}_{n=1}^{\infty}$, or $\{e_{n}\}_{n=1}^{\infty}$, or $\{p_{n}\}_{n=1}^{\infty}$.
In particular, a homomorphism $\theta: \Lambda \to \mathbb{R}$ is uniquely defined by specifying either $\{\theta(h_{n})\}_{n=1}^{\infty}$, or $\{\theta(e_{n})\}_{n=1}^{\infty}$, or $\{\theta(p_{n})\}_{n=1}^{\infty}$. 
Algebra $\Lambda$ as a vector space over $\mathbb{R}$ admits the basis of the {\it Schur functions} $\{S_{\lambda}\}_{\lambda \in \text{ {Partitions}}}$, which can be expressed in terms of the complete symmetric functions via the Jacobi-Trudi identity, \cite[eq.~(3.4) on p.~41]{Mac}: 
\begin{align}
S_{\lambda} = \det \left(h_{\lambda_{i}-i+j} \right)_{1 \leq i, j \leq n} \quad \text{ for any } n \geq \ell(\lambda). 
\end{align}
There is also a larger family of {\it skew Schur functions} $\{S_{\lambda/\mu}\}_{\lambda, \mu \in \text{ Partitions}}$, such that $S_{\lambda/\emptyset} = S_{\lambda}$ and $S_{\lambda/\mu} = 0$ unless $\mu \subset \lambda$. For them one has a more general version of the Jacobi-Trudi identity, \cite[eq.~(5.4) on p.~70]{Mac}:
\begin{align}
\label{skewJacobiTrudi}
S_{\lambda/\mu} = \det \left(h_{\lambda_{i}-\mu_{j}-i+j} \right)_{1 \leq i, j \leq n} \quad \text{ for any } n \geq \ell(\lambda).
\end{align}
Note that matrix $\left(h_{\lambda_{i}-\mu_{j}-i+j} \right)_{1 \leq i, j \leq n}$ after reversing the order of both rows and columns becomes the same as the minor of the infinite matrix $\left(h_{j-i}\right)_{i, j=1}^{\infty}$ specified by choosing columns $\{\lambda_{i}-i+M\}_{i=1}^{n}$ and rows $\{\mu_{j}-j+M\}_{j=1}^{n}$ for any $M$ such that $\mu_{n}-n+M > 0$. Given a sequence $\{a_{n}\}_{n=1}^{\infty}$, we can define a homomorphism $\theta\left[\{a_{n}\}\right]: \Lambda \to \mathbb{R}$ by setting $\theta(h_{n}) = a_{n}$ for all $n \geq 1$, $\theta(1) = 1$. Then \eqref{skewJacobiTrudi} implies
\begin{multline*}
\{a_{n}\}_{n=1}^{\infty} \text{ is a P\'{o}lya frequency sequence } \Longleftrightarrow
\theta\left[\{a_{n}\}\right](S_{\lambda/\mu}) \geq 0 \text{ for any partitions } \mu, \lambda. 
\end{multline*} 
But any skew Schur function can be expressed as a linear combination of Schur functions with non-negative integer (Littlewood-Richardson) coefficients, see \cite[p.~142]{Mac}, so
\begin{multline*}
\{a_{n}\}_{n=1}^{\infty} \text{ is a P\'{o}lya frequency sequence } \Longleftrightarrow
\theta\left[\{a_{n}\}\right](S_{\lambda}) \geq 0 \text{ for any partition } \lambda. 
\end{multline*} 
One can check (see section \ref{sec:Specializations} for details), that the homomorphism $\theta\left[\{a_{n}\}\right]$ with $\{a_{n}\}$ specified by the generating function \eqref{ETgenfun}, can be defined in terms of the power sums by $p_{1} \to \sum_{i=1}^{\infty} \alpha_{i} + \sum_{j=1}^{\infty} \beta_{j} + \gamma$, $p_{k} \to \sum_{i=1}^{\infty} \alpha_{i}^{k} + (-1)^{k-1}\sum_{j=1}^{\infty} \beta_{j}^{k}$ for all $k \geq 2$. Hence Proposition \ref{Edrei-Thoma} can be restated as 
\begin{proposition}[Edrei-Thoma theorem]
\label{Edrei-Thoma2}
A homomorphism $\theta: \Lambda \to \mathbb{R}$ takes non-negative values on all Schur functions if and only if it is defined by
\begin{align}
\theta(p_{1}) = \sum_{i=1}^{\infty} \alpha_{i} + \sum_{j=1}^{\infty} \beta_{j} + \gamma, \quad \theta(p_{k}) = \sum_{i=1}^{\infty} \alpha_{i}^{k} + (-1)^{k-1}\sum_{j=1}^{\infty} \beta_{j}^{k} \text{ for all } k \geq 2,
\end{align}
for some $\alpha_{i}, \beta_{j}, \gamma \geq 0$, such that $\sum_{i =1}^{\infty} \alpha_{i} + \sum_{j =1}^{\infty} \beta_{j} < \infty$.
\end{proposition}
See \cite{BO}, \cite{Me} for a more detailed exposition of the Edrei-Thoma theorem in the context of representations of the infinite symmetric group.

For fixed parameters $q, t$ algebra $\Lambda$ also admits the bases of the {\it symmetric Macdonald functions} $\left \{P_{\lambda}(x_{1}, x_{2}, x_{3}, \ldots; q, t)  \right \}_{\lambda \in \text{Partitions}}$ and $\left \{Q_{\lambda}(x_{1}, x_{2}, x_{3}, \ldots; q, t)  \right \}_{\lambda \in \text{Partitions}}$, that were introduced by I.G.~Macdonald. We have $Q_{\lambda} = b_{\lambda}P_{\lambda}$, where the constant $b_{\lambda}$ is some rational function of $q$ and $t$. For $-1 < q, t < 1$ we have $b_{\lambda}>0$. For $q=t$ both functions $P_{\lambda}$ and $Q_{\lambda}$ become the Schur function $S_{\lambda}$.

Over recent decades the Macdonald polynomials have been an exciting and broad research subject due, in particular, to their deep connections with affine Hecke algebras and Hilbert schemes. S.V.~Kerov has conjectured in \cite[Sec.~7.3]{Ker92}, see also \cite[p.~106]{Ker03}\footnote{There the conjecture is stated in terms of the generating function $\sum_{n=0}^{\infty}\theta(h_{n})z^{n}$, but it is straightforward to check that both formulations are equivalent.}, that it is possible to generalize Proposition \ref{Edrei-Thoma2} to the following theorem.
\begin{theorem}
\label{Kerov conjecture}
For fixed $q, t \in \mathbb{R}$, $|q|, |t| < 1$, a homomorphism $\theta: \Lambda \to \mathbb{R}$ takes non-negative values on all Macdonald functions $Q_{\lambda} (\dots; q, t)$ if and only if 
\begin{align}
\label{Kerov-condition}
\theta(p_{1}) = \sum_{i=1}^{\infty} \alpha_{i} + \frac{1-q}{1-t}\left(\sum_{j=1}^{\infty} \beta_{j} + \gamma\right), \quad \theta(p_{k}) = \sum_{i=1}^{\infty} \alpha_{i}^{k} + (-1)^{k-1}\frac{1-q^{k}}{1-t^{k}} \sum_{j=1}^{\infty} \beta_{j}^{k} \text{ for all } k \geq 2,
\end{align}
for some $\alpha_{i}, \beta_{j}, \gamma \geq 0$, such that $\sum_{i =1}^{\infty} \alpha_{i} + \sum_{j =1}^{\infty} \beta_{j} < \infty$.
\end{theorem}
The main result of this paper is a proof of Theorem \ref{Kerov conjecture}. As in the case of Proposition \ref{Edrei-Thoma2}, verification of the fact that $\theta$ defined by \eqref{Kerov-condition} satisfies $\theta(P_{\lambda}) \geq 0$ for any partition $\lambda$, is relatively straightforward and was known before, see section \ref{sec:Specializations} for details. The hard part is to show that there are no other $\theta$. 

\begin{remark}
Recent interest in homomorphisms with non-negative values on all the Macdonald functions comes, in particular, from the study of Macdonald measures and processes, see \cite{BC}. If $\theta_{1}$, $\theta_{2}$ are two such homomorphisms, then the corresponding Macdonald measure\footnote{This definition works as long as $\displaystyle \sum_{\lambda \in \text{Partitions}}\theta_{1}\left(P_{\lambda}\right)\theta_{2}\left(Q_{\lambda}\right) < \infty$} is a probability measure on partitions that assigns to a partition $\lambda$ probability 
$\sim \theta_{1}\left(P_{\lambda}\right)\theta_{2}\left(Q_{\lambda}\right)$. For $t =0$ and $q \to 1$ such measures (with $\theta_{1}$, $\theta_{2}$ given by particular specifications of \eqref{Kerov-condition}) arise in the study of random polymers, while for $q=0$ and $0 \leq t < 1$ they appear in the study of the stochastic six vertex models, see \cite{BBW}.
\end{remark}

\subsection{Characters of infinite groups and minimal boundaries of branching graphs}
\label{sec:branching}
In this section we briefly review the representation-theoretic significance of
\begin{itemize}
\item
Proposition \ref{Edrei-Thoma2} in the context of the infinite symmetric group $S_{\infty}$.

\item
The Hall-Littlewood case of the Theorem \ref{Kerov conjecture} in the context of infinite matrix groups over finite fields.
\end{itemize}
We also explain how Theorem \ref{Kerov conjecture} can be interpreted as explicit description of the minimal boundary of the Young graph with Macdonald multiplicities.
Consider  the {\it Young graph} $\mathcal{Y}$. Vertices of $\mathcal{Y}$ are partitions graphically represented by Young diagrams (see section \ref{sec:Macdonald Polynomials} for details), and we draw a directed edge $\mu \nearrow \lambda$, if $\lambda$ is obtained from $\mu$ by adding one box, see Fig. \ref{Young}. 
\begin{figure}[h]
\includegraphics[width = 
0.5\textwidth]{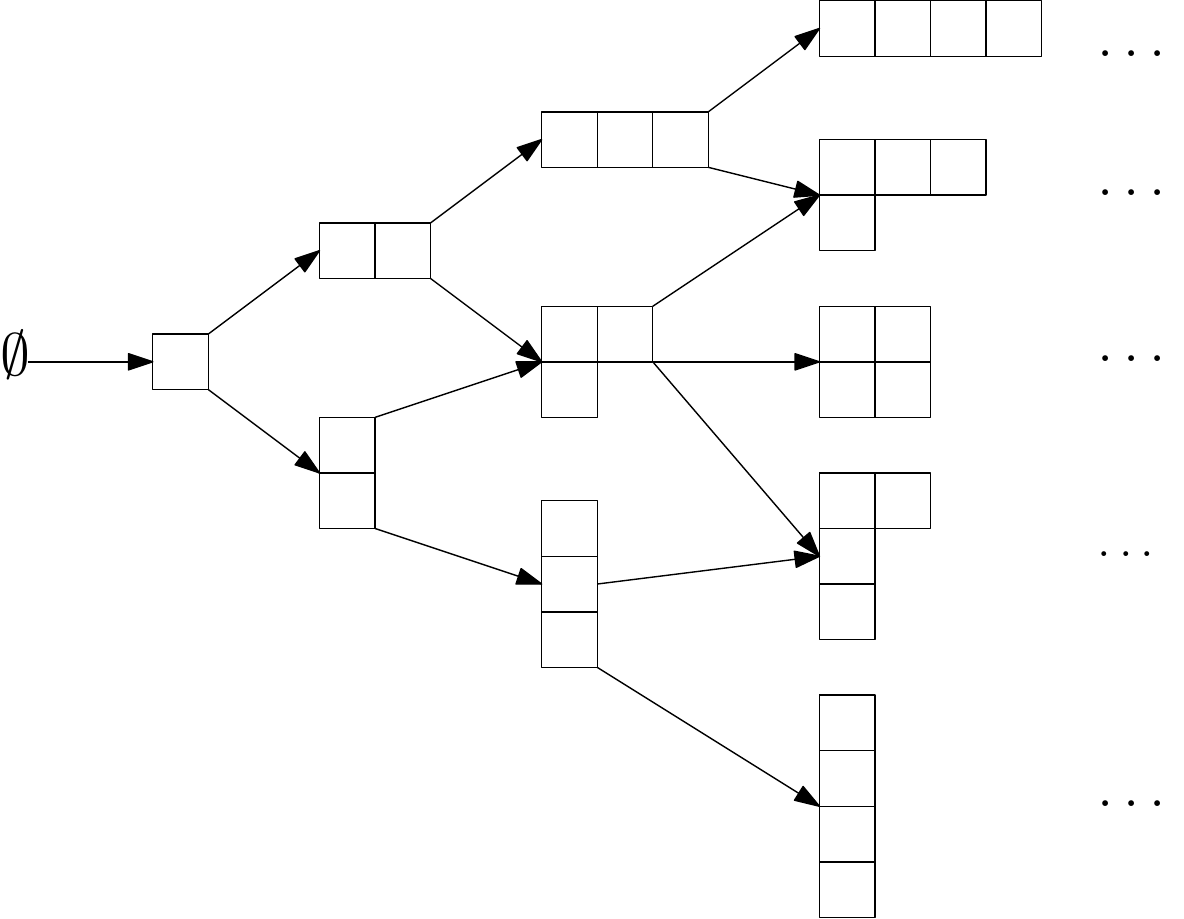}
\caption{Young graph $\mathcal{Y}$.}
\label{Young}
\end{figure}
Then the normalized \footnote{The normalization is $\chi(Id) = 1$.} characters of $S_{\infty}$ can be shown to be in bijection with the normalized non-negative harmonic (in the Vershik-Kerov sense) functions on $\mathcal{Y}$, namely such functions $f: \{\text{Partitions}\} \to \mathbb{R}_{\geq 0}$, that
\begin{enumerate}
\item
$f$ is harmonic in the sense that $f(\mu) = \sum_{\lambda: \mu \nearrow \lambda} f(\lambda)$ for any partition $\mu$.

\item
$f(\emptyset)=1$.
\end{enumerate} 
See \cite[Ch. 3]{BO} for details.
Due to the Pieri formula, \cite[eq.~(5.16) on p.~73]{Mac}: 
\begin{align}
h_{1}S_{\mu} = \sum_{\lambda: \mu \nearrow \lambda}S_{\lambda},
\end{align}
such $f$ gives rise to a linear functional $\widetilde{f}:\Lambda \to \mathbb{R}$ defined by $\widetilde{f}(S_{\lambda}) := f(\lambda)$, such that
\begin{enumerate}
\item
$\widetilde{f}(1)=1$.

\item
$\widetilde{f}\left(S_{\lambda} \right) \geq 0$ for any partition $\lambda$.

\item
$\widetilde{f}\left(h_{1}S_{\lambda}\right) = \widetilde{f}\left(S_{\lambda}\right)$ for any partition $\lambda$.
\end{enumerate}

Denote the set of such functionals by $\mathcal{H}$. It is convex, and the minimal boundary $\mathcal{M}$ of the Young graph $\mathcal{Y}$ is (by definition) the set of extreme points of $\mathcal{H}$. Any functional in $\mathcal{H}$ is an "average" of extreme functionals, i.e. it admits an integral (over $\mathcal{M}$) representation similar to the classical Poisson integral
formula for non-negative harmonic functions on a disk.
The Vershik-Kerov ring theorem, \cite{VK84}, implies that $\widetilde{f} \in \mathcal{H}$ is extreme if and only if $\widetilde{f}$ is a homomorphism. Thus, Proposition \ref{Edrei-Thoma2} allows one to describe $\mathcal{M}$, i.e. the set of extreme characters of $S_{\infty}$, as an infinite-dimensional simplex parametrized by two sequences $\{\alpha_{i}\}_{i=1}^{\infty}$, $\{\beta_{j}\}_{j=1}^{\infty}$, such that $\alpha_{i}, \beta_{j} \geq 0$ and $\displaystyle \sum_{i=1}^{\infty} \alpha_{i} + \sum_{j=1}^{\infty} \beta_{j} \leq 1$.

Using the Macdonald version of the Pieri formula (see section \ref{sec:Macdonald Polynomials} for details)
\begin{align}
\left( \frac{1-t}{1-q}h_{1} \right) Q_{\mu} = \sum_{\lambda: \mu \nearrow \lambda}\psi_{\lambda/\mu}(q, t)Q_{\lambda},
\end{align}
one can for $-1 < q, t < 1$ equip edges $\mu \nearrow \lambda$ of the Young graph $\mathcal{Y}$ with positive multiplicities $\psi_{\lambda/\mu}(q, t)$ and consider the modified question of finding the minimal boundary of such graph $\mathcal{Y}_{q, t}$. More precisely, consider the set $\mathcal{H}_{q, t}$ of non-negative harmonic functions on $\mathcal{Y}_{q, t}$, i.e. such functions $f: \{\text{Partitions}\} \to \mathbb{R}_{\geq 0}$, that $f(\emptyset) = 1$ and 
\begin{align}
f(\mu) = \sum_{\lambda: \mu \nearrow \lambda}\psi_{\lambda/ \mu}(q, t)f(\lambda) \qquad \text{for any partition } \mu.
\end{align}
$\mathcal{H}_{q, t}$ is a convex set. Denote by $\mathcal{M}_{q, t}$ the set of its extreme points (i.e. the minimal boundary of $\mathcal{Y}_{q, t}$).  To each $f \in \mathcal{H}_{q, t}$ corresponds the linear functional $\widetilde{f}: \Lambda \to \mathbb{R}$ defined by $\widetilde{f}(Q_{\mu}) :=f(\mu)$. 
The following proposition shows that Theorem \ref{Kerov conjecture} in fact provides an explicit description of the minimal boundary $\mathcal{M}_{q, t}$.
\begin{proposition}
\label{prop:boundary}
Suppose $-1 < q, t < 1$. Then 
\begin{multline*}
\left\{\widetilde{f} \mid f \in \mathcal{M}_{q, t} \right\} \\ = \left \{ \text{Homomorphisms } \theta:\Lambda \to \mathbb{R} \text{ with } \theta\left(\frac{1-t}{1-q}h_{1}\right) = 1 \text{ and }  \theta \left(Q_{\mu}\right) \geq 0 \text{ for any partition } \mu  \right \}.
\end{multline*}
\end{proposition}
Proposition \ref{prop:boundary} is proved in section \ref{sec:boundary}. The proof is very similar to that of the Vershik-Kerov ring theorem, see  \cite[sec.~8.7]{GO} or \cite[pp. 50-51]{BO} for expositions of the later.  However, we can't apply the Vershik-Kerov theorem directly. Its main condition is the non-negativity of the structure constants of multiplication for the basis of symmetric functions that we are considering.  For Schur functions that condition becomes the  non-negativity of the Littlewood-Richardson coefficients. For Macdonald functions consider the the $q, t$--deformed Littlewood-Richardson coefficients $c^{\nu}_{\lambda, \mu}(q, t)$ defined by 
\begin{align*}
P_{\mu}(q, t)P_{\lambda}(q, t) = \sum_{\nu} c^{\nu}_{\lambda, \mu}(q, t)P_{\nu}(q, t).
\end{align*} 
For $q=0$ and $0 \leq t < 1$ the non-negativity of $c^{\nu}_{\lambda, \mu}$ follows from known formulas for these coefficients, see \cite[Th.~4.9]{Ra}, \cite[Th.~1.3]{Sc}.  Our numerical experiments in \emph{Mathematica } also suggest that $c^{\nu}_{\lambda, \mu}(q, t)$ might indeed be non-negative for both $0 \leq q, t < 1$ and $-1 < q, t \leq 0$. As far as we know, proving non-negativity in such generality is an open question. At the same time, non-negativity doesn't hold for all $(q, t) \in (-1, 1)^{2}$, as can be shown by considering the coefficient $c^{(3, 2, 1)}_{(2, 1),\ (2, 1)}(q, t)$. Our proof of Proposition \ref{prop:boundary} bypasses these complications. Combining Theorem \ref{Kerov conjecture} with Proposition \ref{prop:boundary} we obtain description of the minimal boundary of $\mathcal{Y}_{q, t}$ as an infinite-dimensional simplex parametrized by two sequences $\{\alpha_{i}\}_{i=1}^{\infty}$, $\{\beta_{j}\}_{j=1}^{\infty}$, such that $\alpha_{i}, \beta_{j} \geq 0$ and $\displaystyle \sum_{i=1}^{\infty} \alpha_{i} + \frac{1-q}{1-t} \sum_{j=1}^{\infty} \beta_{j} \leq 1$.

For $q=0$ Macdonald functions become the Hall-Littlewood functions, and the corresponding question of identifying the minimal boundary gets a representation-theoretic meaning in the context of matrix groups over finite fields. Namely, consider the finite field $\mathbb{F}_{p}$, where $p$ is a prime power. Let 
\begin{multline*}
GL(\infty, \mathbb{F}_{p}) := \{ \left[X_{i, j} \right]_{i, j=1}^{\infty} \mid X_{i, j} \in \mathbb{F}_{p}, \text{ and } \exists n, \text{ such that } X^{(n)}:=\left[X_{i, j} \right]_{i, j=1}^{n} \in GL_{n}(\mathbb{F}_{p}) \\ \text{and } X_{i, j} = 1_{i=j} \text{ for } \max\{i, j\} > n \}.
\end{multline*}
Let $\mathbb{U}_{p}$ be the group of infinite upper unitriangular matrices over $\mathbb{F}_{p}$. Equip it with the product topology and the Borel $\sigma$-algebra.  

\begin{definition}
A probability measure $\rho$ on $\mathbb{U}_{p}$ is called {\it central} if $\rho(M) = \rho(gMg^{-1})$ for any measurable $M \subset \mathbb{U}_{p}$ and $g \in GL(\infty, \mathbb{F}_{p})$, such that $gMg^{-1} \subset \mathbb{U}_{p}$. A central measure ̺$\rho$ is called {\it ergodic} if it is an extreme point
of the convex set of all central probability measures.
\end{definition}
In other words, saying that a probability measure $\rho$ is central  is the same as saying that for any $n \times n$ upper unitriangular matrix $A$ the probability of the cylinder set  $\rho \left( \left \{X \in \mathbb{U}_{p} \mid X^{(n)} = A \right \} \right)$ depends only on the conjugacy class of $A$, i.e. only on the partition of $n$ specifying the Jordan normal form of $A$. A certain subclass of central measures corresponds to the unipotent traces of the group of infinite almost upper unitriangular matrices $\mathbb{GLU}$ over $\mathbb{F}_{p}$, see \cite[Sec.~4]{GKV} for details.  Similar to the case of $S_{\infty}$, one can show that the central probability measures are in bijection with the set $\mathcal{H}_{0, 1/p}$. More precisely, if $\rho$ is a central probability measure and $\lambda$ is a partition of size $n$, then the value of the corresponding harmonic function on $\lambda$ is $\rho \left( \left \{X \in \mathbb{U}_{p} \mid X^{(n)} = A \right \} \right) p^{(n(n-1)/2) -\sum (i-1)\lambda_{i}}$ for any $A \in GL(n, \mathbb{F}_{p})$ with Jordan form given by $\lambda$, see \cite{B1} for details.
In other words, the question of classifying  ergodic central measures on $\mathbb{U}_{p}$  can be interpreted as the question of identifying the minimal boundary of $\mathcal{Y}_{0, 1/p}$.
Thus Theorem \ref{Kerov conjecture} together with Proposition \ref{prop:boundary} imply \footnote{In this case we can also apply the Vershik-Kerov ring theorem directly.} the classification of the ergodic central measures on $\mathbb{U}_{p}$, that was conjectured in \cite[conjecture.~4.5]{GKV}. See also \cite{B2},\cite{BuP} for the law of large numbers for such measures. See \cite{V14} for the recent advances in the problem describing central measures on path spaces.

\subsection{Known limiting cases and approaches}
\label{sec:known}
In this section we briefly review the limiting cases of the Kerov conjecture (other than the Schur case $q=t$), that have been proved prior to the current work.  We also review some of the approaches that were suggested over the years to tackle this problem. However, none of those approaches so far have resulted in proving the most general case of  the Kerov conjecture. Classification of the homomorphisms with non-negative values on Macdonald functions has been established before in the following three cases:
\begin{enumerate}
\item
{\it Jack's functions}: $\delta > 0$ is fixed, set $t: = q^{\delta}$, and consider the limit $q \to 1$;

\item
{\it Monomial symmetric functions}: $q=0$, $t=1$;

\item
{\it Schur's $Q$-functions} $q=0$, $t=-1$.

\end{enumerate}

In the Jack's limit $P_{\lambda}(q, t)$ becomes the {\it Jack's symmetric function} $P_{\lambda}^{(1/\delta)}$ (in the notations of \cite{Mac}). Condition \eqref{Kerov-condition} of the Theorem \ref{Kerov conjecture}
then becomes: 
\begin{align*}
\theta(p_{1}) = \sum_{i=1}^{\infty} \alpha_{i} + \delta^{-1}\left(\sum_{j=1}^{\infty} \beta_{j} + \gamma\right), \quad \theta(p_{k}) = \sum_{i=1}^{\infty} \alpha_{i}^{k} + (-1)^{k-1}\delta^{-1}\sum_{j=1}^{\infty} \beta_{j}^{k} \text{ for } k \geq 2,
\end{align*}
for some $\alpha_{i}, \beta_{j}, \gamma \geq 0$, such that $\sum_{i =1}^{\infty} \alpha_{i} + \sum_{j =1}^{\infty} \beta_{j} < \infty$. This limiting case of the Theorem \ref{Kerov conjecture} (as well as identification of the homomorphisms with the minimal boundary of the Jack's graph) was proved in \cite{KOO}. 
The main idea to use the {\it shifted Jack polynomials} to obtain amenable to analysis formula for the relative dimensions in the Jack's graph. This tool was developed in \cite{OO2}, see also \cite{OO1} for the theory of the {\it shifted Schur functions}. The special case $\delta =1$ is again the Edrei-Thoma theorem. The special case $\delta = 1/2$ corresponds to the spherical unitary representations of the Gelfand pair consisting of the "even" infinite symmetric group $S_{2\infty} = \varinjlim S_{2n}$ and its hyperoctahedral subgroup $\varinjlim S_{n} \rtimes \mathbb{Z}_{2}^{n}$. The description of the minimal boundary in this special case was priorly obtained in \cite{Ok97} (see also \cite{Ol}). There also exist more general shifted Macdonald polynomials, see \cite{Ok98}, but it seems that in such generality the connection with the relative dimension is lost. 
%Although the Jack's limiting case of the theorem \ref{Kerov conjecture} doesn't directly follow from the the theorem \ref{Kerov conjecture} itself, we can slightly modify our proofs to make them work for this situation as well.  See section \ref{sec:limiting} for details. However, 
%to identify the Martin boundary of the Jack's graph with the  homomorphisms, that are non-negative on all the Jack's functions, we would once again need to apply the Vershik-Kerov ring theorem. 
The direct application of the Vershik-Kerov ring theorem in the Jack's case would once again  (see remark \ref{rem:LR}) be contigent on the positivity of the Littlewood-Richardson coefficients for the Jack's functions. The later positivity was conjectured in \cite{St}, but to the best of our knowledge is not proved yet.

For $q = 0$ and $t =1$ function $P_{\lambda}$ becomes the {\it monomial symmetric function} $m_{\lambda}: = \sum x^{\sigma}$, where the sum is over all distinct permutations $\sigma$ of $\lambda$. The description of the minimal boundary for $\mathcal{Y}_{0, 1}$ was first obtained in \cite{Kin} by J.F.C~Kingman. He was motivated by a problem of studying random partitions arising in population genetics, and introduced the notion of {\it partition structures}, which are essentially central (i.e constant on conjugacy classes) probability measures on $S_{\infty}$.   Extreme partition structures then correspond to 
homomorphisms $\theta: \Lambda \to \mathbb{R}$, which are non-negative on all the monomial symmetric functions.
In this case one only has the homomorphisms defined by  $x_{i} \to \alpha_{i} \geq 0$ for all $i \geq 1$. See also \cite{Ker89}.

The same desription was obtained for $q = 0$ and $t = -1$ in \cite{N}, see also \cite{Iv}. This case corresponds to projective characters of the infinite symmetric group $S_{\infty}$, which are linearized  by characters of the infinite spin-symmetric group. The Macdonald functions $Q_{\lambda}(0, -1)$ in this case are the Schur's $Q$-functions introduced by I.~Schur in \cite{S}, see \cite[III.8]{Mac} for details.

We would like also to mention another two approaches for proving the Edrei-Thoma case. In \cite{Ok94}, based on \cite{Ol}, the Proposition \ref{Edrei-Thoma2} is proved through description of all spherical representations of a pair $(G, K)$, where $G = S_{\infty} \times S_{\infty}$ is the infinite bisymmetric group and $K$ is its diagonal subgroup. The more recent paper \cite{BuG} proves Proposition \ref{Edrei-Thoma2} by establishing certain stochastic monotonicity in the Young graph $\mathcal{Y}$. It also conjectures that similar stochastic 
monotonicity holds for $\mathcal{Y}_{0, t}$, which would imply the Kerov conjecture for the Hall-Littlewood case. 

\subsection{Synopsis of the proof}
\label{sec:synopsis}
In this section we briefly summarize our proof of Theorem \ref{Kerov conjecture}, which is presented in sections \ref{sec:proof1} and \ref{sec:proof2}. The necessary background on the Macdonald functions is reviewed in sections \ref{sec:Macdonald Polynomials} and \ref{sec:Specializations}. Similar to the proof of Proposition \ref{Edrei-Thoma2} in  \cite{Whit}, \cite{ASW}, \cite{Edr}, as well as the proof of \cite{Th}, our proof is comprised of two parts.
Let
\begin{align}
\label{GeneratingCoefficients0}
g_{r}:=Q_{(r)} = \sum_{(r_{1}, r_{2}, \ldots): \ r_{1}+r_{2}+\cdots = r} \prod_{i \geq 1} \frac{(t; q)_{r_{i}}}{(q, q)_{r_{i}}}x_{i}^{r_{i}}, \quad \text{where} \quad (a; q)_{k}:=\prod_{m=1}^{k}\left(1-aq^{m-1}\right)
\end{align}
is the $q$-{\it Pochhammer symbol} (see section \ref{sec:Macdonald Polynomials} for details).

The first part (see section \ref{sec:proof1}) consists of showing that we can reduce Theorem \ref{Kerov conjecture} to the following generalization of Proposition \ref{ETreduced}.
\begin{theorem}
\label{Kerov conjecture reduced}
For fixed $q, t \in \mathbb{R}$, $|q|, |t| < 1$, if a homomorphism $\tau: \Lambda \to \mathbb{R}$ takes non-negative values on all the Macdonald functions $P_{\lambda}(q, t)$ and satisfies 
\begin{align}
\label{fastdecline1}
\lim_{r \to \infty} \tau(g_{r})^{1/r} = \lim_{r \to \infty} \tau(e_{r})^{1/r} = 0,
\end{align}
 then $\tau(p_{k})=0$ for all $k \geq 2$.
\end{theorem}
To prove this reduction we work with the generating function $\Pi(\theta):= \sum_{r=0}^{\infty} \theta(g_{r})z^{r}$. Proving condition \eqref{Kerov-condition} of Theorem \ref{Kerov conjecture} is equivalent to showing that
\begin{align}
\Pi(\theta) = e^{\gamma z} \cdot \prod_{i=1}^{\infty} \frac{(t \alpha_{i} z; q)_{\infty}}{(\alpha_{i} z; q)_{\infty}} \cdot \prod_{j=1}^{\infty} \left(1 + \beta_{j}z \right),
\end{align}
see section \ref{sec:Specializations} for details. The key result of this part is the "pole removal" Lemma \ref{lemma:poleremoval}, which shows that if $\lim_{r \to \infty} \theta(g_{r+1})/\theta(g_{r}) = \alpha > 0$, then multiplication of $\Pi(\theta)$ by $(\alpha z; q)_{\infty}/(t\alpha z; q)_{\infty}$  produces a generating function of another homomorphism with non-negative values on all the Macdonald functions. This operation allows us to reduce by $1$ the multiplicity of the smallest pole $1/\alpha$ of $\Pi(\theta)$. By a (possibly infinite) sequence of such operations we can remove all the poles of $\Pi(\theta)$. Then, using the duality involution, we can similarly  remove all the zeroes of $\Pi(\theta)$ by factoring out the $1+\beta z$ terms. This completes the reduction of $\theta$ to $\tau$, and of Theorem \ref{Kerov conjecture} to Theorem \ref{Kerov conjecture reduced}. Thus, the first part of our proof develops in the Macdonald setting the analogues of the arguments of \cite{Whit} and \cite{ASW}. On the way we have to overcome some technical difficulties specific to the Macdonald case, such as showing that sequence $\{\theta(g_{r})\}$ doesn't behave in a "wild" way, see Lemma \ref{lemma:limit}. The key tool repeatedly used throughout the first part of the proof is the Pieri formula \eqref{Pieri}.

The second part of our proof (see section \ref{sec:proof2}) is devoted to proving Theorem \ref{Kerov conjecture reduced}. It is not clear how to generalize the Nevanlinna theory approach of \cite{Edr} to the Macdonald setting. This is an interesting problem on its own, but we have found a different route, which is a combination of using the Pieri formulas \eqref{Pieri}, \eqref{Pieridual}, and soft combinatorial methods. For the Schur case $q=t$ this approach also gives a new proof of Proposition \ref{ETreduced}. In a nutshell, the proof of Theorem \ref{Kerov conjecture reduced} works as follows.
 
Suppose $k \geq 2$ is the smallest, such that $\tau(p_{k}) \neq 0$. Then $\tau(g_{k}) \neq \tau(g_{1})^{k}/k!$, and so for any partition $\mu$ the absolute value of the difference between $\tau\left(g_{k}Q_{\mu}\right)$ and $\tau\left(g_{1}^{k}Q_{\mu}/k!\right)$ is of order $\tau(Q_{\mu})$. On the other hand, with the use of the Pieri formula we can expand both $\tau\left(g_{k}Q_{\mu}\right)$ and $\tau\left(g_{1}^{k}Q_{\mu}/k!\right)$ as weighted sums of $\tau(Q_{\lambda})$ over those $\lambda$, which are obtained by adding $k$ boxes to $\mu$. We then compare these expansions term by term and find out that the coefficients  of $\tau(Q_{\lambda})$ in both are close, if the $k$ boxes of $\lambda\backslash \mu$ are far apart from each other. To obtain contradiction we would like to show that in fact the difference between $\tau\left(g_{k}Q_{\mu}\right)$ and $\tau\left(g_{1}^{k}Q_{\mu}/k!\right)$ is of order smaller than $\tau(Q_{\mu})$, at least for some partition $\mu$. Thus it will be sufficient to deduce from \eqref{fastdecline1} that (at least for some $\mu$) most contributions to $\tau\left(g_{k}Q_{\mu}\right)$ and $\tau\left(g_{1}^{k}Q_{\mu}/k!\right)$ come from terms with the $k$ boxes of $\lambda\backslash \mu$ far apart from each other. In a way, we need to show that certain "diffusivity" takes place in the Young graph with the Macdonald multiplicities. The key result of the second part is proving this "diffusivity", see Lemma \ref{lemma:badsteps}.

How to find such partition $\mu$, for which the "diffusivity" holds? We expect that in fact it holds for a typical partition of large enough size. Since such partition has many outer corners,  it is reasonable to expect that the $k$ added boxes will most likely be far apart from each other. The proof of Lemma \ref{lemma:badsteps} is essentially a more delicate version of choosing a large random partition, i.e. an application of the probabilistic method in combinatorics. See Lemmas \ref{lemma:2tbaleauxbound} and \ref{lemma:comparison} for details.

\subsection{Acknowledgments} The author wishes to express his gratitude to Alexei Borodin and Vadim Gorin for useful comments and discussions.

\section{Macdonald functions}
\label{sec:Macdonald Polynomials}
This section is a brief review of facts and notations concerning the Macdonald functions. It is based on \cite[Ch.~VI]{Mac}, and we mostly follow notations of this book. 
A {\it partition} $\lambda$ is a non-increasing finite sequence $\lambda_{1} \geq \lambda_{2} \geq \cdots \geq \lambda_{\ell}$ of positive integers. Parameter $\ell = \ell(\lambda)$ is called the {\it length} of $\lambda$. We set $\lambda_{i}=0$ for $i>\ell$. Let $|\lambda| := \lambda_{1} + \cdots + \lambda_{\ell}$ be the {\it size} of $\lambda$. It is convenient to represent a partition by its {\it Young diagram} -- a left-justified array of boxes with $\lambda_{i}$ boxes in the $i$-th row. We will use the so called English way to depict Young diagrams (see Fig. \ref{YD}). We will use the term "partition" also for its Young diagram. For $i \leq \ell(\lambda)$ and $j \leq \lambda_{i}$ denote by $(i, j)$ the $j$-th box in the $i$-th row of $\lambda$. Denote by $\lambda'_{j}$ the length of the $j$-th column of $\lambda$. Partition $\lambda' := \left(\lambda'_{1}, \lambda'_{2}, \ldots, \lambda'_{\lambda_{1}}\right)$ is called the {\it conjugate} partition. We will use the following notations:
\begin{enumerate}
\item
$\mu \subset \lambda$ if $\ell(\mu) \leq \ell(\lambda)$ and $\mu_{i} \leq \lambda_{i}$ for any $1 \leq i \leq \ell(\mu)$, i.e. each box of $\mu$ is contained in $\lambda$.
\item
$\mu \prech \lambda$ if $\mu \subset \lambda$ and $\lambda\backslash\mu$ is a {\it horizontal strip}, i.e. has at most one box in each column.
\item
$\mu \precv \lambda$ if $\mu \subset \lambda$ and $\lambda\backslash\mu$ is a {\it vertical strip}, i.e. has at most one box in each row.
\end{enumerate}
\begin{figure}[h]
\label{YD}
\includegraphics[width = 
0.4\textwidth]{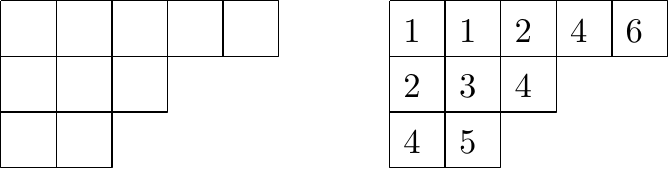}
\caption{Left: Young diagram corresponding to partition $(5, 3, 2)$. Right: a semistandard tableau of shape $(5, 3, 2)$.}
\end{figure}
A {\it semistandard tableau} of shape $\lambda$ is a filling of boxes of $\lambda$ with positive integers that is weakly increasing along each row and strictly increasing down each column (see Fig. \eqref{YD}). Denote by $\mathcal{SS}(\lambda)$ the set of such fillings.  For a tableau $T$ we denote by $T_i$ the shape of the subtableau formed by boxes with entries $\leq i$. The condition of being  semistandard is equivalent to saying that each $T_i$ is a valid partition and $T_{i-1} \prech T_{i}$ for any $i \geq 1$. We will call the sequence $\left \{|T_{i}|-|T_{i-1}|\right \}_{i=1}^{\infty}$ the {\it content} of $T$. Note that $|T_{i}|-|T_{i-1}|$ is the number of entries of $T$ equal to $i$.  A semistandard tableau of shape $\lambda$ is called {\it standard} if  each $1 \leq k \leq |\lambda|$ appears exactly once as its entry. Denote by $\mathcal{ST}(\lambda)$ the set of standard tableaux of shape $\lambda$.  We will use $Sh(T)$ to denote the shape of tableau $T$. For $\mu \subset \lambda$ we can also define semistandard and standard {\it skew tableaux} of shape $\lambda / \mu$ as fillings of boxes of $\lambda \backslash \mu$ with integers, that respectively satisfy the same properties, as for the usual tableaux. Denote by $\mathcal{ST}(\lambda/\mu)$ the set of standard skew tableaux of shape $\lambda/\mu$. 

\smallskip

Fix two parameters $q, t$. For a partition $\lambda$ and its box $s=(i, j)$ let 
\begin{align}
b_{\lambda}(s) = b_{\lambda}(i, j):= \frac{1-q^{\lambda_{i}-j}t^{\lambda'_{j}-i+1}}{1-q^{\lambda_{i}-j+1}t^{\lambda'_{j}-i}} \quad \text{\and} \quad b_{\lambda} := \prod_{s \in \lambda} b_{\lambda}(s).
\end{align}
If a box $s$ is outside of $\lambda$, then set $b_{\lambda}(s) := 1$. For $\mu \subset \lambda$ denote by $R_{\lambda/\mu}$ (respectively by $C_{\lambda/\mu}$) the union of all rows (respectively columns) containing  boxes from $\lambda \backslash \mu$. For partitions $\mu \prech \lambda$ let
\begin{align}
\label{PsiPhi}
\psi_{\lambda/\mu} := \prod_{s \in R_{\lambda/\mu} - C_{\lambda/\mu}} \frac{b_{\mu}(s)}{b_{\lambda}(s)} \quad \text{and} \quad \phi_{\lambda/\mu} := \prod_{s \in C_{\lambda/\mu}} \frac{b_{\lambda}(s)}{b_{\mu}(s)},
\end{align}
For partitions $\mu \precv \lambda$ let
\begin{align}
\label{PsiPhiPrime}
\psi'_{\lambda/\mu} := \prod_{s \in C_{\lambda/\mu} - R_{\lambda/\mu}} \frac{b_{\lambda}(s)}{b_{\mu}(s)} \quad \text{and} \quad \phi'_{\lambda/\mu} := \prod_{s \in R_{\lambda/\mu}} \frac{b_{\mu}(s)}{b_{\lambda}(s)}.
\end{align}
For a tableau $T$ let 
\begin{align}
\psi(T) := \prod_{i \geq 1} \psi_{T_{i}/T_{i-1}} \quad \text{and} \quad \phi(T) := \prod_{i \geq 1} \phi_{T_{i}/T_{i-1}}. 
\end{align}
For a partition $\lambda$ define the {\it Macdonald symmetric functions}
\begin{align}
\label{tableausum}
P_{\lambda} = \sum_{T \in \mathcal{SS}(\lambda)}  \psi(T) x^{T} \quad \text{and} \quad Q_{\lambda} = \sum_{T \in \mathcal{SS}(\lambda)} \phi(T) x^{T},
\end{align}
where $\displaystyle x^{T}:= \prod_{i \geq 1} x_{i}^{|T_{i}|-|T_{i-1}|}$. Then one can show that $Q_{\lambda} = b_{\lambda}P_{\lambda}$. Note that $P_{(1^{r})}=e_{r}$. For $q=t$ we have $\psi(T) = \phi(T)=1$, so in this case $P_{\lambda} = Q_{\lambda} = S_{\lambda}$, where $S_{\lambda}$ is the Schur function. Let
\begin{align}
\label{GeneratingCoefficients}
g_{r}:=Q_{(r)} = \sum_{(r_{1}, r_{2}, \ldots): \ r_{1}+r_{2}+\cdots = r} \prod_{i \geq 1} \frac{(t; q)_{r_{i}}}{(q, q)_{r_{i}}}x_{i}^{r_{i}}, \quad \text{where} \quad (a; q)_{k}:=\prod_{m=1}^{k}\left(1-aq^{m-1}\right)
\end{align}
is the $q$-{\it Pochhammer symbol}. Then one can show that $\Lambda = \mathbb{R}[g_{1}, g_{2}, g_{3}, \ldots]$. Moreover, $\Lambda$ is freely generated as an $\mathbb{R}$-algebra by $\{g_{n}\}_{n=1}^{\infty}$. Set $g_{0}:=1$. For a partition $\lambda$ set $\displaystyle g_{\lambda}:= \prod_{i \geq 1} g_{\lambda_{i}}$. 

\smallskip

The most important tool for us will be the Pieri formulas:
\begin{align}
\label{Pieri}
Q_{\mu}g_{r} = \sum_{\mu \prech \lambda, \ |\lambda|-|\mu|=r} \psi_{\lambda/\mu}Q_{\lambda} \quad \text{and} \quad P_{\mu}g_{r} = \sum_{\mu \prech \lambda, \ |\lambda|-|\mu|=r} \phi_{\lambda/\mu}P_{\lambda}; \\
\label{Pieridual}
Q_{\mu}e_{r} = \sum_{\mu \precv \lambda, \ |\lambda|-|\mu|=r} \phi'_{\lambda/\mu}Q_{\lambda} \quad \text{and} \quad P_{\mu}e_{r} = \sum_{\mu \precv \lambda, \ |\lambda|-|\mu|=r} \psi'_{\lambda/\mu}P_{\lambda}
\end{align}
for any partition $\lambda$ and $r \geq 1$.
One can define an automorphism $\omega_{q, t}$ of $\Lambda$ by setting $\omega_{q, t}(g_{r}) = e_{r}$ for all $r \geq 1$. It has the property that for any partition $\lambda$
\begin{align}
\label{Duality}
\omega_{q, t}(Q_{\lambda}(x; q, t)) = P_{\lambda'}(x; t, q) \quad \text{and} \quad
\omega_{q, t}(P_{\lambda}(x; q, t)) = Q_{\lambda'}(x; t, q).
\end{align}
In particular, $\omega_{t, q} \circ \omega_{q, t} = Id$. It can also be shown that $\omega_{q, t}(p_{r}) = (-1)^{r-1}\frac{1-q^{r}}{1-t^{r}}p_{r}$. 

\smallskip

For partitions $\mu \subset \lambda$ one can also define {\it skew Macdonald functions} $P_{\lambda\slash\mu}$ and $Q_{\lambda \slash\mu}$ by using formulas \eqref{tableausum} for sums over semistandard skew  tableaux of shape $\lambda\slash\mu$. In this definition we take $T_{i}$ to be the shape of the union of $\mu$ and the subtableau of $T$ comprised of boxes with entries $\leq i$. In particular, $P_{\lambda \slash \emptyset}=P_{\lambda}$ and $Q_{\lambda \slash \emptyset}=Q_{\lambda}$. We also have
\begin{align}
\label{skewDuality}
\omega_{q, t}(Q_{\lambda \slash \mu}(x; q, t)) = P_{\lambda' \slash \mu'}(x; t, q) \quad \text{and} \quad
\omega_{q, t}(P_{\lambda \slash \mu}(x; q, t)) = Q_{\lambda' \slash \mu'}(x; t, q).
\end{align}
The $q$-Gauss summation formula implies that
\begin{align}
\label{qGauss}
\frac{(az; q)_{\infty}}{(z; q)_{\infty}}= \sum_{n=0}^{\infty}\frac{(a; q)_{n}}{(q; q)_{n}}z^{n}
\end{align}
as formal power series. It follows from \eqref{qGauss} and \eqref{GeneratingCoefficients} that
\begin{align}
\sum_{n=0}^{\infty} g_{n}z^{n} = \prod_{i \geq 1} \frac{(tzx_{i}; q)_{\infty}}{(zx_{i}; q)_{\infty}}.
\end{align}
It can also be shown that 
\begin{align}
\label{powerexponent}
\sum_{n=0}^{\infty} g_{n}z^{n} = \exp\left(\sum_{n=1}^{\infty} \frac{1}{n} \frac{1-t^{n}}{1-q^{n}}p_{n}z^{n}\right)
\end{align} 
as formal power series.

\section{Positive Specializations}
\label{sec:Specializations}
From now on we assume that $q, t \in \mathbb{R}$ and $|q|, |t| < 1$. A {\it specialization} is a homomorphism $\theta: \Lambda \to \mathbb{R}$. It is said to be $(q, t)$-{\it Macdonald-positive} if $\theta(P_{\lambda})$ is  non-negative for any partition $\lambda$ (equivalently, $\theta(Q_{\lambda})= b_{\lambda}\theta(P_{\lambda})$ is non-negative for any partition $\lambda$). If $\theta$ is a $(q, t)$- Macdonald-positive specialization, then $w_{t, q}(\theta):=\theta \circ \omega_{t, q}$ is a $(t, q)$-Macdonald-positive specialization due to \eqref{Duality}. We will call such $w_{t, q}(\theta)$ the {\it dual} of $\theta$. Clearly, $w_{q, t}(w_{t, q}(\theta)) = \theta$. For a specialization $\theta$ we define its {\it generating function} $\Pi(\theta) = \Pi_{q, t} (\theta)$ as the formal power series $\displaystyle \sum_{n=0}^{\infty} \theta(g_{n})z^{n}$. Given two specializations $\theta_{1}, \theta_{2}$ one can define their union $\theta = (\theta_{1}, \theta_{2})$ by setting $\theta(p_{n}) := \theta_{1}(p_{n})+\theta_{2}(p_{n})$ for all $n \geq 1$. Then by \eqref{powerexponent} we have $\Pi((\theta_{1}, \theta_{2})) = \Pi(\theta_{1})\Pi(\theta_{2})$, which could also be taken as an alternative definition of $(\theta_{1}, \theta_{2})$. One can show that
\begin{align}
\label{SpecUnion}
(\theta_{1}, \theta_{2})\left(P_{\lambda/\mu}\right) = \sum_{\nu: \ \mu \subset \nu \subset \lambda} \theta_{1}\left(P_{\lambda/\nu}\right)\theta_{2}\left(P_{\nu/\mu}\right) \quad \text{and} \quad 
(\theta_{1}, \theta_{2})\left(Q_{\lambda/\mu}\right) = \sum_{\nu: \ \mu \subset \nu \subset \lambda} \theta_{1}\left(Q_{\lambda/\nu}\right)\theta_{2}\left(Q_{\nu/\mu}\right).
\end{align}
Note that 
\begin{align}
\label{dualunion}
w_{q, t}((\theta_{1}, \theta_{2})) = (w_{q, t}(\theta_{1}), w_{q, t}(\theta_{2})).
\end{align}
The following specializations are $(q, t)$-Macdonald-positive and, in fact, satisfy $\theta\left(Q_{\lambda/\mu}\right) \geq 0$ for any partitions $\mu \subset \lambda$:
\begin{enumerate}
\item
$\tau_{\alpha}$ for any $\alpha \geq 0$ defined by setting $x_{1} \to \alpha$, $x_{i} \to 0$ for $i \geq 2$. In other words, $\tau_{\alpha}(p_{n}) = \alpha^{n}$. Positivity follows from \eqref{tableausum}.  More precisely, 
\begin{align}
\tau_{\alpha}\left(Q_{\lambda/\mu}\right) =
\begin{cases}
\phi_{\lambda/\mu}(q, t)\alpha^{|\lambda|-|\mu|}, & \quad \text{if }\mu \prech \lambda, \\
0, & \quad \text{otherwise}.
\end{cases}
\end{align}
$\tau_{\alpha}(g_{n}) = \frac{(t; q)_{n}}{(q; q)_{n}}\alpha^{n}$, so by \eqref{qGauss} the generating function is $\Pi(\tau_{\alpha}) = (t\alpha z; q)_{\infty}/{(\alpha z; q)_{\infty}}$. 

\smallskip

\item
$w_{q, t}(\tau_{\beta})$ for any $\beta \geq 0$. By \eqref{Duality} we have
\begin{align}
w_{q, t}(\tau_{\beta})\left(Q_{\lambda/\mu}\right) =
\begin{cases}
\psi_{\lambda'/\mu'}(t, q)\beta^{|\lambda|-|\mu|}, & \quad \text{if }\mu \precv \lambda, \\
0, & \quad \text{otherwise}.
\end{cases}
\end{align}
This specialization can be defined by sending $p_{n} \to (-1)^{n-1}\frac{1-q^{n}}{1-t^{n}}\beta^{n}$ and hence has the generating function $\displaystyle \Pi(w_{q, t}(\tau_{\beta})) = \exp \left( \sum_{n=1}^{\infty} (-1)^{n-1}\beta^{n}z^{n}/n\right) = 1+\beta z$ by \eqref{powerexponent}. 

\smallskip

\item
$\displaystyle \tau_{Pl, \gamma}: = \lim_{m \to \infty} T_{m, \gamma}$ for any $\gamma \geq 0$, where
\begin{align*}
T_{m, \gamma}:=\underbrace{\left(\tau_{(1-q)\gamma/((1-t)m)}, \ \ldots, \ \tau_{(1-q)\gamma/((1-t)m)}\right)}_{\text{union of } m \text{ specializations}}.
\end{align*}
The specialization $\tau_{Pl, \gamma}$ can be defined by sending $p_{1} \to \frac{1-q}{1-t} \gamma$, $p_{n} \to 0$ for $n \geq 2$, hence by \eqref{powerexponent} it has the generating function $\Pi(\tau) = e^{\gamma z}$. Specialization $\tau_{Pl, \gamma}$ is called the {\it Plancherel} specialization with parameter $\gamma$. Denote by $\mathcal{SS}(\lambda/\mu, m)$ the set of semistandard skew tableaux of shape $\lambda/\mu$ with entries $1, 2, \ldots, m$. Denote by $\mathcal{S}(\lambda/\mu, k)$ the set of semistandard skew tableaux $T$ of shape $\lambda/\mu$ with entries $1, 2, \ldots, k$, each of which appears at least once.  Then by \eqref{SpecUnion} 
\begin{multline}
T_{m, \gamma}(Q_{\lambda/\mu}) = \left(\frac{1-q}{1-t} \gamma \right)^{|\lambda|-|\mu|} \left(\frac{1}{m} \right)^{|\lambda|-|\mu|}\sum_{T \in \mathcal{SS}(\lambda/\mu, m)} \phi_{T}(q, t) \\ = 
 \left(\frac{1-q}{1-t} \gamma \right)^{|\lambda|-|\mu|} \left(\frac{1}{m} \right)^{|\lambda|-|\mu|}\sum_{k=0}^{|\lambda|-|\mu|}\left(\begin{array}{c} m \\ k \end{array} \right) \cdot \sum_{T \in \mathcal{S}(\lambda/\mu, k)} \phi_{T}(q, t).
\end{multline}
Taking limit $m \to \infty$ and using the fact that $\displaystyle \lim_{m \to \infty} \left(\frac{1}{m} \right)^{|\lambda|-|\mu|}\left(\begin{array}{c} m \\ k \end{array} \right) = 0$ for $k < |\lambda|-|\mu|$ we get
\begin{align}
\tau_{Pl, \gamma}(Q_{\lambda/\mu}) =  \frac{1}{(|\lambda|-|\mu|)!} \cdot \left(\frac{1-q}{1-t} \gamma \right)^{|\lambda|-|\mu|} \sum_{T \in \mathcal{ST}(\lambda/\mu)} \phi_{T}(q, t).
\end{align}
\end{enumerate}

It follows from \eqref{SpecUnion} that any union of finitely many specializations of the above-defined three types is also $(q, t)$-Macdonald-positive. By taking limit we get the following proposition.
\begin{proposition}
Suppose $q, t \in \mathbb{R}$, $|q|, |t| < 1$. If a specialization $\theta$ is defined by the generating function
\begin{align}
\label{MainGenFun}
\Pi(\theta) = e^{\gamma z} \cdot \prod_{i=1}^{\infty} \frac{(t \alpha_{i} z; q)_{\infty}}{(\alpha_{i} z; q)_{\infty}} \cdot \prod_{j=1}^{\infty} \left(1 + \beta_{j}z \right)
\end{align}
for some $\alpha_{i}, \beta_{j}, \gamma \geq 0$, such that $\displaystyle \sum_{i=1}^{\infty} \alpha_{i} + \sum_{j=1}^{\infty} \beta_{j} < \infty$, then $\theta$ is $(q, t)$-Macdonald-positive.
\end{proposition}
It is clear from the above arguments that condition \eqref{Kerov-condition} of Theorem \ref{Kerov conjecture} is equivalent to \eqref{MainGenFun}. Hence to prove Theorem \ref{Kerov conjecture} we need to show that if $\theta$ is a $(q, t)$-Macdonald-positive specialization, then \eqref{MainGenFun} holds for some appropriate choice of $\alpha_{i}, \beta_{j}, \gamma$. 

\begin{remark}
\label{rem:convergence} 
Note that $\displaystyle \sum_{i=1}^{\infty} \alpha_{i} < \infty$ implies the convergence to a nonzero limit of $\displaystyle \prod_{i=1}^{\infty} \frac{(t \alpha_{i} z; q)_{\infty}}{(\alpha_{i} z; q)_{\infty}}$ for any $z \in \mathbb{C}$, such that $|z| < 1/\alpha_{1}$.
Indeed, by taking logarithm and using the limit comparison test, we know that absolute convergence of $\displaystyle \sum_{n=1}^{\infty} p_{n}$ for $p_{n} \in \mathbb{C}$, $|p_{n}| < 1$ implies convergence to a nonzero limit of $\displaystyle \prod_{n=1}^{\infty} (1+p_{n})$.  And  $\displaystyle \sum_{i=1}^{\infty} \alpha_{i} < \infty$ implies $\displaystyle \sum_{i=1}^{\infty} \sum_{k=0}^{\infty} \left \lvert \frac{1-tq^{k}\alpha_{i}z}{1-q^{k}\alpha_{i}z} - 1 \right \rvert < \infty$, since $|q| < 1$. This convergence is uniform on compact sets, so it also implies coefficientwise convergence of power series. 
\end{remark}

\section{Proof. Part I: Pole removal}
\label{sec:proof1}
\begin{proof}[Proof of Theorem \ref{Kerov conjecture}]

Let $\theta$ be a $(q, t)$-Macdonald positive specialization. Observe that $\theta(g_{n}) = \theta(Q_{(n)}) \geq 0$.
\begin{lemma}
\label{lemma:zerovalue}
If $\theta(Q_{\mu}) = 0$, then $\theta(Q_{\lambda}) = 0$ for any $\lambda \supset \mu$.
\end{lemma}
\begin{proof}[Proof of Lemma \ref{lemma:zerovalue}]
By \eqref{Pieri} we can express $0 = \theta(g_{1}Q_{\mu})$ as a sum of non-negative terms $\psi_{\lambda/\mu}\theta(Q_{\lambda})$ over all partitions $\lambda$ that are obtained by adding a box to $\mu$. Since $\psi_{\lambda/\mu} > 0$, we get $\theta(Q_{\lambda}) = 0$ for each such $\lambda$. Proceeding by induction we get $\theta(Q_{\lambda}) = 0$ for any $\lambda \supset \mu$.
\end{proof}
So if $\theta(g_{m}) = 0$ for some $m$, then $\theta(g_{n}) = 0$ for all $n > m$, hence $\Pi(\theta)$ is a polynomial in $z$.  
\begin{lemma}
\label{lemma:psibound}
For fixed $q, t \in (-1; 1)$ there exists such $C > 1$, that $C^{-k} < \psi_{\lambda\slash\mu} < C^{k}$ for any $\mu \prech \lambda$ with $R_{\lambda\slash\mu}-C_{\lambda\slash\mu}$ having at most $k$ rows and $C^{-k} < \phi'_{\lambda\slash\mu} < C^{k}$ for any $\mu \precv \lambda$ with $R_{\lambda\slash\mu}$ having at most $k$ rows. 
\end{lemma}
\begin{proof}[Proof of Lemma \ref{lemma:psibound}]
By \eqref{PsiPhi} $\psi_{\lambda\slash\mu}$ is a product of a most $k$ terms, each of which is a product of expressions $b_{\mu}(s)/b_{\lambda}(s)$ over a subset of boxes of a single row of $\lambda$. The product of numerators of  
$b_{\mu}(s)$ over a subset of boxes of a single row can be bounded from below by $(\max\{|q|, |t|\}; |q|)_{\infty}$ and from above by $(-\max\{|q|, |t|\}; |q|)_{\infty}$. Same holds for the product of denominators of  
$b_{\mu}(s)$, the product of numerators of  
$b_{\lambda}(s)$, and the product of denominators of  $b_{\lambda}(s)$. Hence for this case we could take any $C > (-\max\{|q|, |t|\}; |q|)_{\infty}^{2}/(\max\{|q|, |t|\}; |q|)_{\infty}^{2}$. A similar argument holds for $\phi'_{\lambda\slash\mu}$.
%\begin{align}
%C > (-\max\{|q|, |t|\}; \max\{|q|, |t|\})_{\infty}^{2}/(\max\{|q|, |t|\}; \max\{|q|, |t|\})_{\infty}^{2}.
%\end{align}
\end{proof}
We reserve $C$ to denote this constant for the rest of the proof.
\begin{lemma}
\label{lemma:limit}
If $\theta$ is $(q, t)$-Macdonald-positive and $\theta(g_{n}) > 0$ for all $n$, then either 
\begin{enumerate}
\item
$\displaystyle \lim_{n\to \infty} \theta(g_{n+1})/\theta(g_{n})$ exists and is finite and positive, or
\item
There exists such $s \in \mathbb{Z}_{>0}$, that $\displaystyle \lim_{n\to \infty} \theta(g_{n+s})/\theta(g_{n}) = 0$.
\end{enumerate}
\end{lemma}
\begin{proof}[Proof of Lemma \ref{lemma:limit}]
Suppose that neither of the two stated properties holds.
By \eqref{Pieri} we have 
\begin{align}
\label{TwoRowSinglePieri}
\theta(g_{n}g_{1}) = \frac{(1-t)\left(1-q^{n+1}\right)}{(1-q)\left(1-tq^{n}\right)}\theta(g_{n+1}) + \theta(Q_{(n, 1)}),
\end{align} hence we have $0 < \theta(g_{n+1})/\theta(g_{n}) \leq \frac{\theta(g_{1})(1-q)(1+|tq|))}{(1-t)(1-q^{2})}$, so the sequence $\{\theta(g_{n+1})/\theta(g_{n})\}_{n=1}^{\infty}$ is bounded. By our assumption it has at least two subsequential limit points. Then there exists some $\gamma > 1$, such that for any $N$ one can choose $N < k < n$, so that $\frac{\theta(g_{n+1})\theta(g_{k-1})}{\theta(g_{n})\theta(g_{k})} > \gamma$. By \eqref{Pieri} we have 
\begin{align}
\label{TwoRowPieri1}
\theta\left(g_{n}g_{k}\right) = \theta\left(Q_{(n, k)}\right) + \sum_{\ell=1}^{k} \frac{(t; q)_{n-k+\ell}\left(q^{\ell+1}; q\right)_{n-k+\ell}}{(q; q)_{n-k+\ell}\left(tq^{\ell}; q\right)_{n-k+\ell}}\theta\left(Q_{(n+\ell, k-\ell)}\right), \\
\label{TwoRowPieri2}
\theta(g_{n+1}g_{k-1}) = \sum_{\ell=1}^{k} \frac{(t; q)_{n-k+\ell+1}\left(q^{\ell}; q\right)_{n-k+\ell+1}}{(q; q)_{n-k+\ell+1}\left(tq^{\ell-1}; q\right)_{n-k+\ell+1}}\theta \left(Q_{(n+\ell, k-\ell)}\right),
\end{align}
where both sums on the right hand side are comprised of non-negative terms. Note that 
\begin{align}
\label{coefficientratio}
\frac{\text{Coefficient of } \theta\left(Q_{(n+\ell, k-\ell)}\right) \text{ in \eqref{TwoRowPieri1}}}{\text{Coefficient of } \theta\left(Q_{(n+\ell, k-\ell)}\right) \text{ in \eqref{TwoRowPieri2}}}=\frac{\left(1-q^{n-k+\ell+1}\right)\left(1-tq^{\ell-1}\right)}{\left(1-tq^{n-k+\ell}\right)\left(1-q^{\ell}\right)}. 
\end{align}
So for $q \geq \max\{t, 0\}$ the right hand side of \eqref{coefficientratio} would always be $\geq 1$, hence \eqref{TwoRowPieri1} would be greater or equal than \eqref{TwoRowPieri2}, so we would obtain a contradiction. However, we also need to find an argument that works for other pairs $(q, t) \in (-1, 1)^{2}$.  To obtain a contradiction it will be enough to show that ratio of \eqref{TwoRowPieri1} and \eqref{TwoRowPieri2} will be close to $1$ for large enough $n$, $k$. The ratio \eqref{coefficientratio} is close to $1$ for large enough $\ell$, so we just need to show that the contribution of the terms with small $\ell$ will be small. We can choose $L$ such that the ratio \eqref{coefficientratio} is greater than $\gamma^{-1/2}$ for all $\ell > L$ and all $n > k \geq 1$. We will now show that for any fixed $s > 0$  
\begin{align}
\label{inequality0}
\theta\left(Q_{(n_{1}+s, n_{2}-s)}\right) \geq \frac{1}{2C^{8}}\theta\left(Q_{(n_{1}, n_{2})}\right)
\end{align}
for all sufficiently large $n_{2} < n_{1}$. It will imply that for some fixed $M$ the  terms in \eqref{TwoRowPieri1} and \eqref{TwoRowPieri2} with $L+M \geq \ell > L$  will be for large enough $n, k$ at least as large, up to a constant, as terms with $\ell \leq L$. This will allow us to obtain a contradiction.

\smallskip

Let $\displaystyle \epsilon_{s}:= \limsup_{m\to \infty} \theta(g_{m+s})/\theta(g_{m})$. By our assumption $\epsilon_{s}>0$ for any $s \in \mathbb{Z}_{>0}$. Note that $\displaystyle \epsilon_{s}^{-1} = \liminf_{m\to \infty} \theta(g_{m})/\theta(g_{m+s})$. By Lemma \ref{lemma:limit} we have $C^{-2} < \psi_{\lambda/\mu} < C^{2}$ for any partitions $\mu \prech \lambda$, $\ell(\mu) \leq 2$. 
Consider $\theta\left(Q_{(n_{1}+s, n_{2}-s)}g_{m+s}\right)$ and $\theta\left(Q_{(n_{1}+s, n_{2})}g_{m}\right)$ and with the help of $\eqref{Pieri}$ expand both expressions as linear combinations of $\theta(Q_{\lambda})$ with non-negative coefficients. For $m \leq n_{2}-s$ all $\theta(Q_{\lambda})$ that appear in the expansion of $\theta\left(Q_{(n_{1}+s, n_{2})}g_{m}\right)$ also appear in the expansion of $\theta\left(Q_{(n_{1}+s, n_{2}-s)}g_{m+s}\right)$. Indeed, if  $\theta(Q_{\lambda})$ appears in the expansion of $\theta\left(Q_{(n_{1}+s, n_{2})}g_{m}\right)$, then $(n_{1}+s, n_{2}) \prech \lambda$, $\ell(\lambda) \leq 3$ and $\lambda_{3} \leq n_{2}-s$, so $(n_{1}+s, n_{2}-s) \prech \lambda$. Hence $\theta\left(Q_{(n_{1}+s, n_{2}-s)}g_{m+s}\right) \geq C^{-4}\theta\left(Q_{(n_{1}+s, n_{2})}g_{m}\right)$, so
\begin{align}
\label{inequality1}
\theta\left(Q_{(n_{1}+s, n_{2}-s)}\right) \geq C^{-4}\max_{1 \leq m \leq n_{2}-s}\left \{\frac{\theta(g_{m})}{\theta(g_{m+s})} \right \}\theta\left(Q_{(n_{1}+s, n_{2})}\right) \geq \frac{1}{2} C^{-4} \epsilon_{s}^{-1}\theta\left(Q_{(n_{1}+s, n_{2})}\right) 
\end{align}
for all large enough $n_2$. Now consider $\theta\left(Q_{(n_{1}+s, n_{2})}g_{p}\right)$ and $\theta\left(Q_{(n_{1}, n_{2})}g_{p+s}\right)$ and with the help of $\eqref{Pieri}$ expand both expressions as linear combinations of $\theta(Q_{\lambda})$ with non-negative coefficients. For $p \geq n_{1}$ all $\theta(Q_{\lambda})$ that appear in the expansion of $\theta\left(Q_{(n_{1}, n_{2})}g_{p+s}\right)$ also appear in the expansion of $\theta\left(Q_{(n_{1}+s, n_{2})}g_{p}\right)$. Indeed, if $\theta(Q_{\lambda})$ appears in the expansion of $\theta\left(Q_{(n_{1}, n_{2})}g_{p+s}\right)$, then $(n_{1}, n_{2}) \prech \lambda$, $\ell(\lambda) \leq 3$ and $\lambda_{1} = n_{1} + (p+s) - (\lambda_{2} - n_{2}) - \lambda_{3} \geq p+s \geq n_{1}+s$, so $(n_{1}+s, n_{2}) \prech \lambda$. Hence $\theta\left(Q_{(n_{1}+s, n_{2})}g_{p}\right) \geq C^{-4}\theta\left(Q_{(n_{1}, n_{2})}g_{p+s}\right)$, so
\begin{align}
\label{inequality2}
\theta\left(Q_{(n_{1}+s, n_{2})}\right) \geq C^{-4}\sup_{p \geq n_{1}}\left \{\frac{\theta(g_{p+s})}{\theta(g_{p})} \right \}\theta\left(Q_{(n_{1}, n_{2})}\right) \geq C^{-4}\epsilon_{s}\theta\left(Q_{(n_{1}, n_{2})}\right).
\end{align}
\eqref{inequality0} follows by combining \eqref{inequality1} and \eqref{inequality2}. 

\smallskip

Take $M >2LC^{12}\left(\gamma^{1/2}-1\right)^{-1}$. By applying inequality \eqref{inequality0} with $s = L, 2L, \ldots, \lfloor M/L \rfloor L$ together with Lemma \ref{lemma:psibound} we get that for sufficiently large $k <n$
\begin{multline}
\label{smallvsbigterms}
\sum_{\ell=1}^{L} \psi_{(n+\ell, k-\ell)/(n+1)}\theta\left(Q_{(n+\ell, k-\ell)}\right) \leq \left(\gamma^{1/2}-1\right)\sum_{\ell=L+1}^{L+M} \psi_{(n+\ell, k-\ell)/(n+1)}\theta\left(Q_{(n+\ell, k-\ell)}\right), \text{ hence} \\
\sum_{\ell=1}^{L} \psi_{(n+\ell, k-\ell)/(n+1)}\theta\left(Q_{(n+\ell, k-\ell)}\right) \leq \left(1 -\gamma^{-1/2} \right)\sum_{\ell=1}^{k} \psi_{(n+\ell, k-\ell)/(n+1)}\theta\left(Q_{(n+\ell, k-\ell)}\right).
\end{multline} 
By combining \eqref{TwoRowPieri1}, \eqref{TwoRowPieri2}, \eqref{smallvsbigterms} and the way that $L$ was chosen, we get that for sufficiently large $k <n$
\begin{multline*}
\theta\left(g_{n}g_{k}\right) \geq \sum_{\ell=L+1}^{k} \frac{(t; q)_{n-k+\ell}\left(q^{\ell+1}; q\right)_{n-k+\ell}}{(q; q)_{n-k+\ell}\left(tq^{\ell}; q\right)_{n-k+\ell}}\theta\left(Q_{(n+\ell, k-\ell)}  \right) \\ \geq \gamma^{-1/2} \sum_{\ell=L+1}^{k} \frac{(t; q)_{n-k+\ell+1}\left(q^{\ell}; q\right)_{n-k+\ell+1}}{(q; q)_{n-k+\ell+1}\left(tq^{\ell-1}; q\right)_{n-k+\ell+1}}\theta \left(Q_{(n+\ell, k-\ell)}\right) \\ \geq \gamma^{-1} \sum_{\ell=1}^{k} \frac{(t; q)_{n-k+\ell+1}\left(q^{\ell}; q\right)_{n-k+\ell+1}}{(q; q)_{n-k+\ell+1}\left(tq^{\ell-1}; q\right)_{n-k+\ell+1}}\theta \left(Q_{(n+\ell, k-\ell)}\right) = \gamma^{-1} \theta\left(g_{n+1}g_{k-1}\right).
\end{multline*}
Contradiction.
\end{proof}

If for some $s \in \mathbb{Z}_{>0}$ we have $\displaystyle \lim_{n\to \infty} \theta(g_{n+s})/\theta(g_{n}) = 0$, then the series $\Pi(\theta)$ converges absolutely for every  $z \in \mathbb{C}$ and so is an entire function of $z$. If $\displaystyle \lim_{n\to \infty} \theta(g_{n+1})/\theta(g_{n}) = \alpha > 0$, then the series $\Pi(\theta)$ converges absolutely on the open disk $D = \{z: |z| < 1/\alpha\}$, and so gives a holomorphic function on $D$. 
To deal with the possibility of singularity at $1/\alpha$ we will need the following "pole removal" lemma.

\begin{lemma}["Pole removal" lemma]
\label{lemma:poleremoval}
If $\theta$ is $(q, t)$-Macdonald-positive and $\displaystyle \lim_{n\to \infty} \theta(g_{n+1})/\theta(g_{n}) = \alpha > 0$, then the specialization $\tilde{\theta}$ defined by the generating series  
\begin{align}
\Pi\left(\tilde{\theta}\right) = \Pi(\theta) (\alpha z; q)_{\infty}/ (t \alpha z; q)_{\infty} 
\end{align}
is also $(q, t)$-Macdonald-positive.
\end{lemma}
\begin{proof}[Proof of Lemma \ref{lemma:poleremoval}]
We need to show that $\tilde{\theta}(Q_{\lambda}) \geq 0$ for any partition $\lambda$. We will do it by showing that
\begin{align}
\label{MainLimit}
\tilde{\theta}(Q_{\lambda}) = \lim_{N \to \infty} \frac{\theta\left(Q_{(N) \cup \lambda}\right)}{\theta(g_{N})}, 
\end{align}
where $(N) \cup \lambda$ is the partition obtained by prepending to $\lambda$ a first row of length $N$ (for sufficiently large $N$). Fix $\lambda$. Let $F: \Lambda \to \Lambda$ be a homomorphism defined by replacing $x_{1}$ with $\alpha$, and $x_{i}$ with $x_{i-1}$ for each $i \geq 2$. Note that $F(p_{n}) = \alpha^{n}+p_{n}$, so $F$ is invertible. We will show that both sides of \eqref{MainLimit} are equal to $\theta(F^{-1}(Q_{\lambda}))$.
By \eqref{GeneratingCoefficients} we have $ \displaystyle F(g_{n}) = \sum_{k=0}^{n} \frac{(t; q)_{k}}{(q; q)_{k}}\alpha^{k} g_{n-k}$, so by \eqref{qGauss} we get
\begin{align}
\sum_{n=0}^{\infty} F(g_{n})z^{n} = \left( \sum_{n=0}^{\infty}g_{n}z^{n}\right)\frac{(t\alpha z; q)_{\infty}}{(\alpha z; q)_{\infty}}.
\end{align}
Hence after multiplying both sides by $(\alpha z; q)_{\infty}/(t \alpha z; q)_{\infty}$ and applying $\theta \circ F^{-1}$ we get 
\begin{align}
\left( \sum_{n=0}^{\infty}\theta(g_{n})z^{n}\right)\frac{(\alpha z; q)_{\infty}}{(t\alpha z; q)_{\infty}} = \sum_{n=0}^{\infty} \theta(F^{-1}(g_{n}))z^{n},
\end{align}
so $\tilde{\theta} = \theta \circ F^{-1}$, hence  $\tilde{\theta}(Q_{\lambda}) =\theta(F^{-1}(Q_{\lambda}))$.

For partitions $\nu, \mu$ of the same size denote by $\mathcal{S}(\nu, \mu)$ the set of all semistandard tableaux of shape $\nu$ and content $\mu$. Let
\begin{align}
c(\nu, \mu):= \sum_{T \in \mathcal{S}(\nu, \mu)} \psi(T).
\end{align}
If $c(\nu, \mu) \neq 0$, then $\nu$ {\it dominates} $\mu$, i.e.  $\nu_{1} + \cdots + \nu_{i} \geq  \mu_{1} + \cdots + \mu_{i}$ for all $i \geq 1$. These inequalities must hold, since all entries $\leq i$ in any semistandard tableau must be located in the first $i$ rows of this tableau. Also $c(\nu, \nu)=1$, since there is only one tableau $T$ (with $\psi(T)=1$) of both shape $\nu$ and content $\nu$: one in which all entries in row $i$ are equal to $i$ for any $i \geq 1$. By a repeated application of  \eqref{Pieri} we get
\begin{align}
g_{\mu} = \sum_{\nu} c(\nu, \mu)Q_{\nu}, 
\end{align}
Take a linear order on the set of all partitions, such that $\nu < \mu$, if either $|\nu| < |\mu|$, or $|\nu| = |\mu|$ and $\nu$ is greater than $\mu$ in the lexicographic order. Then define (with respect to this order) an infinite matrix $M$ by setting $M_{\nu, \mu} := c(\nu, \mu)$. Matrix $M$ is upper unitriangular, since if $\nu$ dominates $\mu$, then $\nu$ is greater or equal than $\mu$ in the lexicographic order. So $M$ is invertible and  $M^{-1}$ is also upper unitriangular. For $N > |\lambda|$ denote by $\Delta(N)$ the ordered set of all partitions of size $N+|\lambda|$ with the first row of length $\geq N$ (again, with respect to the considered order). Denote by $\Delta$ the ordered set of partitions of size $\leq |\lambda|$ (with respect to the considered order). Let $M(N)$ be a finite square submatrix of $M$ with rows and columns corresponding to $\Delta(N)$. Since we have a natural order preserving bijection $\Delta(N) \to \Delta$ (deletion of the first row), we can slightly abuse notation by working with each $M(N)$ as with a $\Delta \times \Delta$ matrix. Let $A$ be a $\Delta \times \Delta$ diagonal matrix with $A_{\nu, \mu} := \mathsf{1}_{\nu=\mu} \cdot \alpha^{-|\nu|}$. 
For any linear operator $T: V \to V$ and any finite linearly independent sets of vectors $\mathcal{B}, \mathcal{D}$ in a vector space $V$, such that $T(span(\mathcal{B})) \subset span(\mathcal{D})$, denote by $[T]_{\mathcal{B}, \mathcal{D}}$ the matrix of $T \vert_{span(\mathcal{B})}$ with respect to bases $\mathcal{B}$ of the domain and $\mathcal{D}$ of the codomain. We will first show that
\begin{align}
\label{matrixlimit}
\lim_{N \to \infty} AM(N)A^{-1} = \left[F\right]_{\{g_{\nu}: \ \nu \in \Delta\}, \{Q_{\nu}:\ \nu \in \Delta \}}
\end{align}
Indeed, for $\nu, \mu \in \Delta$ we have 
\begin{align}
\left(AM(N)A^{-1}\right)_{\nu, \mu} = \alpha^{|\mu|-|\nu|}\sum_{T \in \mathcal{S}(N)}\psi(T), \quad \text{where} \ \mathcal{S}(N) := \mathcal{S}((N+|\lambda|-|\nu|) \cup \nu, (N+|\lambda|-|\mu|) \cup \mu).
\end{align} 
Note that $\mathcal{S}(N)$ is empty (and the corresponding matrix element is $0$) unless $(N+|\lambda|-|\nu|) \cup \nu$ dominates $(N+|\lambda|-|\mu|) \cup \mu$ (and, in particular, $|\nu| \leq |\mu|$). For a semistandard tableau $T$ denote by $T^{-}$ the tableau obtained from $T$ by deleting the first row and replacing each entry $i$ by $i-1$. Denote by $T^{1}$ the (row) tableau formed by the first row of $T$.  Denote by $T^{1-}$ the (row) tableau obtained from $T^1$ by deleting all boxes with entries equal to $1$ and replacing each entry $i \geq 2$ by $i-1$. See Figure \ref{fig:longrow} for an example of $T$, $T^{-}$, $T^{1-}$.

\begin{figure}[h]

\includegraphics[width = 
0.7\textwidth]{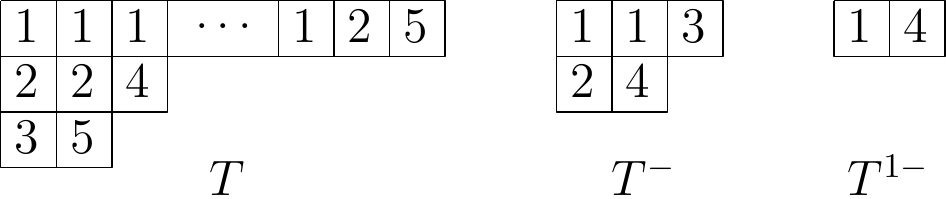}
\caption{An example of $T$, $T^{-}$, $T^{1-}$. Here 
$S_{1} = \{(1^{\text{st}} \text{ box},\ 2), \ (1^{\text{st}} \text{ box},\ 3), \ (2^{\text{nd}} \text{ box},\ 2), \ (2^{\text{nd}} \text{ box},\ 5), \ 
(3^{\text{rd}} \text{ box},\ 4) \}
$,
$
S_{2} = \{(1^{\text{st}} \text{ box},\ 4), \ (1^{\text{st}} \text{ box},\ 5), \ (2^{\text{nd}} \text{ box},\ 3), \ (2^{\text{nd}} \text{ box},\ 4), \ 
(3^{\text{rd}} \text{ box},\ 2), \
(3^{\text{rd}} \text{ box},\ 3), \
(3^{\text{rd}} \text{ box},\ 5) \}
$.
}
\label{fig:longrow}
\end{figure}
Denote by $\widetilde{\mathcal{S}}$ the set of pairs of semistandard tableaux $(T_{1}, T_{2})$ of 
shape $\nu$ and $(|\mu|-|\nu|)$ respectively  and with total content $\mu$. The map $T \to \left(T^{-}, T^{1-}\right)$ provides a bijection $\mathcal{S}(N) \to \widetilde{\mathcal{S}}$ for all large enough $N$.
Denote by $S_1$ the set  all such pairs $(s, i)$, that $s$ is a box in the first row of $T$ with entry $1$, and $i \geq 2$ is an entry in the column of $s$. Denote by $S_2$ the set  all such pairs $(s, i)$, that $s$ is a box in the first row of $T$ with entry $1$, and $i \geq 2$ is not an entry in the column of $s$, while the length of this column is greater or equal than $2$ and $i$ appears somewhere in $T$. 
Note that by \eqref{PsiPhi}
\begin{align}
\psi(T) = \psi\left(T^{-}\right)\psi\left(T^{1}\right)\Pi_{1} \Pi_{2}, \quad \text{where } 
\Pi_{1}:=\prod_{(s, i) \in S_{1}} \frac{b_{T^{1}_{i}}(s)}{b_{T^{1}_{i-1}}(s)}, \quad \Pi_{2} = \prod_{(s, i) \in S_{2}} \frac{b_{T^{1}_{i}}(s)b_{T_{i-1}}(s)}{b_{T^{1}_{i-1}}(s)b_{T_{i}}(s)}.
\end{align}
$\Pi_{1}, \Pi_{2} \to 1$ as $N \to \infty$, so 
\begin{multline}
\lim_{N \to \infty} \left(AM(N)A^{-1}\right)_{\nu, \mu} = \alpha^{|\mu|-|\nu|}\lim_{N \to \infty} \sum_{T \in \mathcal{S}(N)} \psi\left(T^{-}\right)\psi\left(T^{1}\right) \\ = \alpha^{|\mu|-|\nu|}\lim_{N \to \infty} \sum_{T \in \mathcal{S}(N)} \psi\left(T^{-}\right)\phi\left(T^{1-}\right) \frac{(q; q)_{N+|\lambda|-|\nu|} (t; q)_{N+|\lambda|-|\mu|}}{(t; q)_{N+|\lambda|-|\nu|}(q; q)_{N+|\lambda|-|\mu|}} \\ = \lim_{N \to \infty} \alpha^{|\mu|-|\nu|}\sum_{T \in \mathcal{S}(N)} \psi\left(T^{-}\right)\phi\left(T^{1-}\right) 
= \alpha^{|\mu|-|\nu|}\sum_{(T_{1}, T_{2}) \in \widetilde{\mathcal{S}}} \psi\left(T_{1}\right)\phi\left(T_{2}\right).
\end{multline}
By \eqref{tableausum} and \eqref{Pieri} it is equal to the coefficient of $\displaystyle \prod_{i \geq 1} x_{i}^{\mu_{i}}$ in  
\begin{align}
\label{matrixproduct}
\alpha^{|\mu|-|\nu|} P_{\nu}g_{|\mu|-|\nu|} = \sum_{\nu \prech \chi, \ |\chi| = |\mu|} \alpha^{|\chi|-|\nu|}\phi_{\chi/\nu}P_{\chi}, \quad
\text{which is} \quad 
\sum_{\nu \prech \chi, \ |\chi| = |\mu|} \alpha^{|\chi|-|\nu|} \phi_{\chi/\nu}c(\chi, \mu).
\end{align}

Note that all $\chi$ in the right hand side sum are contained in $\Delta$, since $|\chi|= |\mu| \leq |\lambda|$. Let $B$ be a $\Delta \times \Delta$ matrix with $B_{\nu, \chi} := 1_{\nu \prech \chi} \cdot \alpha^{|\chi|-|\nu|} \phi_{\chi/\nu}$. For a symmetric polynomial in $x_{1}, \ldots, x_{n}$ replacing $x_{1}$ with $\alpha$ and $x_{i}$ with $x_{i-1}$ for $2 \leq i \leq n$ amounts to the same result as just replacing $x_{n}$ with $\alpha$.  So by \eqref{tableausum} and \eqref{matrixproduct}
\begin{multline*}
B=\left[F\right]_{\{Q_{\nu}: \ \nu \in \Delta\}, \{Q_{\nu}:\ \nu \in \Delta \}}, \quad \text{and} \lim_{N \to \infty} AM(N)A^{-1} \\ = B M \vert_{\Delta \times \Delta} = \left[F\right]_{\{Q_{\nu}: \ \nu \in \Delta\}, \{Q_{\nu}:\ \nu \in \Delta \}}[Id]_{\{g_{\nu}: \ \nu \in \Delta\}, \{Q_{\nu}:\ \nu \in \Delta \}} = [F]_{\{g_{\nu}: \ \nu \in \Delta\}, \{Q_{\nu}:\ \nu \in \Delta \}}.
\end{multline*} 
This proves \eqref{matrixlimit}. Hence
\begin{align}
\lim_{N \to \infty} AM(N)^{-1}A^{-1} = \left[F^{-1}\right]_{\{Q_{\nu}: \ \nu \in \Delta\}, \{g_{\nu}:\ \nu \in \Delta \}}.
\end{align}
So we get
\begin{multline}
\label{limitevaluation}
\lim_{N \to \infty} \frac{\theta\left(Q_{(N) \cup \lambda}\right)}{\theta(g_{N})} =\lim_{N \to \infty} \sum_{\nu \in \Delta} \left(M(N)^{-1}\right)_{\nu, \lambda}\frac{\theta(g_{\nu})\theta(g_{N+|\lambda|-|\nu|})}{\theta(g_{N})} \\ = \lim_{N \to \infty} \sum_{\nu \in \Delta} \left(M(N)^{-1}\right)_{\nu, \lambda}\theta(g_{\nu}) \alpha^{|\lambda|-|\nu|} = \theta \left(\lim_{N \to \infty} \sum_{\nu \in \Delta} \left(AM(N)^{-1}A^{-1}\right)_{\nu, \lambda} g_{\nu} \right) = \theta(F^{-1}(Q_{\lambda})),
\end{multline}
where the second equality follows, since $\displaystyle \lim_{n \to \infty} \frac{\theta(g_{n+s})}{\theta(g_{n})}= \alpha^{s}$ for any $s \in \mathbb{Z}_{>0}$. This finishes the proof of \eqref{MainLimit}.
\end{proof}
Suppose that $\theta$ is a $(q, t)$-Macdonald-positive specialization. By combining Lemma \ref{lemma:limit} and Lemma \ref{lemma:poleremoval} we can define a sequence $\{\theta_{k}\}$ of $(q, t)$-Macdonald-positive specializations, such that: 
\begin{enumerate}
\item
$\theta_{1} = \theta$.

\item
The sequence is either infinite or has the last term $\theta_{N}$, such that $\Pi\left(\theta_{N}\right)$ is an entire function.

\item
$\displaystyle \lim_{n\to \infty} \frac{\theta_{k}(g_{n+1})}{\theta_{k}(g_{n})} = \alpha_{k} > 0$ for $1 \leq k \leq N-1$ (or any $k \geq 1$ if the sequence is infinite).

\item
$\displaystyle \Pi\left(\theta_{k+1}\right) = \frac{\Pi(\theta_{k}) (\alpha_{k} z; q)_{\infty}}{(t \alpha_{k} z; q)_{\infty}}$ for $1 \leq k \leq N-1$ (or any $k \geq 1$ if the sequence is infinite).

\item 
$\alpha_{1} \geq \alpha_{2} \geq \cdots$.
\end{enumerate}
Suppose that the sequence is infinite. Then $0 \leq \theta_{k+1}(g_{1}) = \theta(g_{1}) - \frac{1-t}{1-q}(\alpha_{1} + \cdots + \alpha_{k})$ for any $k\geq1$, so  $\displaystyle \sum_{k=1}^{\infty} \alpha_{k}$ converges. Hence $\displaystyle \Pi(\theta) \cdot \prod_{k=1}^{\infty} \frac{(\alpha_{k} z; q)_{\infty}}{(t \alpha_{k} z; q)_{\infty}}$ is a well defined power series, which gives an entire function. Indeed, it is equal to $\displaystyle \Pi(\theta_{\ell}) \cdot \prod_{k=\ell+1}^{\infty} \frac{(\alpha_{k} z; q)_{\infty}}{(t \alpha_{k} z; q)_{\infty}}$ and the later expression converges for all $z \in \mathbb{C}$, $|z|<1/\alpha_{\ell}$, see Remark \ref{rem:convergence}. Since $\displaystyle \lim_{\ell \to \infty} \alpha_{\ell} = 0$, it in fact converges for all $z$. Then $\displaystyle \Pi(\theta) \cdot \prod_{k=1}^{\infty} \frac{(\alpha_{k} z; q)_{\infty}}{(t \alpha_{k} z; q)_{\infty}} = \Pi(r(\theta))$ for some specialization $r(\theta)$, which is also $(q, t)$-Macdonald -positive, since $\displaystyle r(\theta)(Q_{\lambda}) = \lim_{k \to \infty} \theta_{k}(Q_{\lambda})$ for any partition $\lambda$. In case of a finite sequence set $r(\theta):=\theta_{N}$. It follows, in particular, that $\Pi(\theta)$ can be analytically continued to a meromorphic function. 

\smallskip

Consider then the $(t,q)$-Macdonald-positive specialization $\bar{\theta} = r(w_{t, q}(r(\theta)))$.  Let $\{\alpha_{i}\}_{i=1}^{\infty}$ and $\{\beta_{j}\}_{j=1}^{\infty}$ be sequences of non-negative numbers with $\displaystyle \sum_{i=1}^{\infty} \alpha_{i} + \sum_{j=1}^{\infty} \beta_{j} < \infty$, such that 
\begin{align}
\label{parametersout}
\Pi_{q, t}(\theta) = \Pi_{q, t}(r(\theta))  \cdot \prod_{i=1}^{\infty} \frac{(t \alpha_{i} z; q)_{\infty}}{( \alpha_{i} z; q)_{\infty}} \quad \text{and} \quad \Pi_{t, q}(w_{t, q}(r(\theta))) = \Pi_{t, q}\left(\bar{\theta}\right)  \cdot \prod_{j=1}^{\infty} \frac{(q\beta_{j} z; t)_{\infty}}{( \beta_{j} z; t)_{\infty}}.
\end{align} 
For an arbitrary series $A(z) = 1+\sum_{n=1}^{\infty}a_{n} z^{n}$ we will write $w_{q, t}(A(z))$ for the  generating series  $\Pi_{q, t}(w_{q, t}(\tau))$, where $\tau$ is the specialization defined by $\Pi_{t, q}(\tau) = A(z)$. In particular, $w_{q, t}\left(\Pi_{t, q}\left(\hat{\theta}\right)\right) = \Pi_{q, t}(w_{q, t}\left(\hat{\theta}\right))$ for any specialization $\hat{\theta}$. If $A(z) = \exp \left(\sum_{m=1}^{\infty}h_{m}z^{m} \right)$, then by $\eqref{powerexponent}$  we get $\tau(p_{m}) = (-1)^{m-1}\frac{1-t^{m}}{1-q^{m}}h_{m}$ for the corresponding specialization $\tau$, hence $w_{q, t}(\tau)(p_{m}) = h_{m}$, hence 
\begin{align}
\label{dualonseries2}
w_{q, t} \left(\exp\left(\sum_{m=1}^{\infty}h_{m}z^{m}\right) \right) = \exp\left(\sum_{m=1}^{\infty}(-1)^{m-1}\frac{1-t^{m}}{1-q^{m}}h_{m}z^{m}\right).
\end{align}
Also by \eqref{dualunion} we have $w_{q, t} \left(A(z)B(z)\right) = w_{q, t}(A(z))w_{q, t}(B(z))$ for any series $A(z)$, $B(z)$.  As we have seen in Section \ref{sec:Specializations},
\begin{align}
\label{dualonseries}
w_{q, t}\left(\frac{(q \beta z; t)_{\infty}}{(\beta z; t)_{\infty}}\right) = 1+\beta z,  \qquad w_{q, t} \left(1+\beta z \right) = \frac{(t \beta z; q)_{\infty}}{(\beta z; q)_{\infty}},  
\end{align}
Then by applying $w_{q, t}$ we get by \eqref{parametersout} and \eqref{dualonseries}  that
\begin{multline}
\label{PlancherelReduction}
\Pi_{q, t}(r(\theta)) = w_{q, t}\left(\Pi_{t, q}(w_{t, q}(r(\theta)))\right) =  w_{q, t}\left(\Pi_{t, q}\left(\bar{\theta}\right)\right) \cdot \prod_{j=1}^{\infty} (1 + \beta_{j}z), \\
\Pi_{q, t}(\theta) = \Pi_{q, t}(r(\theta))  \cdot \prod_{i=1}^{\infty} \frac{(t \alpha_{i} z; q)_{\infty}}{( \alpha_{i} z; q)_{\infty}} = \Pi_{q, t}\left(w_{q, t}\left(\bar{\theta}\right)\right)  \cdot \prod_{i=1}^{\infty} \frac{(t \alpha_{i} z; q)_{\infty}}{( \alpha_{i} z; q)_{\infty}} \cdot \prod_{j=1}^{\infty} (1+\beta_{j}z).
\end{multline}
$\Pi_{t, q}\left(\bar{\theta}\right)$ is an entire function with value $1$ at $z=0$. It also has no zeroes. Indeed, suppose that $\delta$ is its zero of the smallest possible absolute value. Suppose $\ell$ is a multiplicity of $\delta$. Note that $\delta \notin \mathbb{R}_{\geq 0}$, since $\Pi_{t, q}\left(\bar{\theta}\right)(z) \in \mathbb{R}_{\geq 1}$ for $z \in \mathbb{R}_{\geq 0}$. Then $\Pi_{t, q}\left(\bar{\theta}\right)  = (1-\delta^{-1}z)^{\ell}H(z)$ for an entire function $H(z)$, such that $H(0)=1$ and $H(\delta) \neq 0$.  On the open disk $D = \{z: |z| < |\delta|\}$ we have $H(z) = e^{h(z)}$  for some convergent $\displaystyle h(z) = \sum_{m=1}^{\infty} h_{m}z^{m}$. 
Then by applying $w_{q, t}$  and \eqref{PlancherelReduction}, \eqref{dualonseries2} we have for $z \in D$.
\begin{multline*}
\Pi_{q, t}(r(\theta))(z) = \left(\prod_{j=1}^{\infty} (1+\beta_{j}z)\right) w_{q, t}\left(\Pi_{t, q}\left(\bar{\theta}\right)\right) = \left(\prod_{j=1}^{\infty} (1+\beta_{j}z)\right) w_{q, t} \left( 1-\delta^{-1}z \right)^{\ell}w_{q, t}\left(e^{h(z)}\right)  \\ = \left(\prod_{j=1}^{\infty} (1+\beta_{j}z)\right)\frac{(-t\delta^{-1}z; q)^{\ell}_{\infty}}{(-\delta^{-1}z; q)^{\ell}_{\infty}}\left(\exp\left(\sum_{m=1}^{\infty} (-1)^{m-1}\frac{1-t^{m}}{1-q^{m}}h_{m}z^{m}\right)\right) \\
= \left(\prod_{j=1}^{\infty} (1+\beta_{j}z)\right)\frac{(-t\delta^{-1}z; q)_{\infty}^{\ell}}{(-\delta^{-1}z; q)_{\infty}^{\ell}}H(-z)^{-1}\left(\exp\left(\sum_{m=1}^{\infty} (-1)^{m-1}\frac{q^{m}-t^{m}}{1-q^{m}}h_{m}z^{m}\right)\right),
\end{multline*}
but the right hand side goes to $\infty$ as  $z \to -\delta$, $z \in D$. Indeed, in such limit transition $(1+\delta^{-1}z)^{\ell} \to 0$, and the product of all other terms goes to a finite nonzero limit, since $1 - \beta_{j} \delta \neq 0$ for all $j \geq 1$, $H(\delta) \neq 0$, and the radius of convergence of $\displaystyle \sum_{m=1}^{\infty} (-1)^{m-1}\frac{q^{m}-t^{m}}{1-q^{m}}h_{m}z^{m}$ is at least $|\delta|/\max\{|q|, |t|\} > |\delta|$. This is a contradiction with the fact that $\Pi_{q, t}(r(\theta))$ is entire.  Hence $\Pi_{t, q}\left(\bar{\theta}\right)$ has no zeroes.
So $\Pi_{t, q}\left(\bar{\theta}\right) = e^{h(z)}$ for some entire function $\displaystyle h(z) = \sum_{m=1}^{\infty} h_{m}z^{m}$. Then $\tau: = w_{q, t}\left(\bar{\theta} \right)$ is a $(q, t)$-Macdonald positive specialization, such that
\begin{align*}
\Pi_{q, t}(\tau) = w_{q, t} \left( \Pi_{t, q} \left(\bar{\theta}\right) \right) = \exp \left(\sum_{m=1}^{\infty}(-1)^{m-1}\frac{1-t^{m}}{1-q^{m}}h_{m}z^{m}\right)
\end{align*} 
has no poles (or zeroes).
Hence
\begin{align}
\label{fastdecline}
\lim_{r \to \infty} \tau(g_{r})^{1/r} = \lim_{r \to \infty} \tau(e_{r})^{1/r} = 0.
\end{align}
Indeed, \eqref{fastdecline} follows, since we know that both
\begin{align*}
 \displaystyle \Pi_{q, t}(\tau) = \sum_{r=0}^{\infty} \tau(g_{r})z^{r} \quad \text{and}  \quad \displaystyle \sum_{r=0}^{\infty} \tau(e_{r})z^{r} = \sum_{r=0}^{\infty} \bar{\theta}(g_{r}(t, q))z^{r} = \Pi_{t, q}\left(\bar{\theta}\right)
\end{align*}
  converge for any $z$.  By \eqref{PlancherelReduction} it remains to prove that $\tau$ is a Plancherel specialization. So we have reduced proving Theorem  \ref{Kerov conjecture} to proving Theorem \ref{Kerov conjecture reduced}, and it remains to prove the later theorem.
\end{proof}  

\section{Proof. Part II. Diffusivity Argument}
\label{sec:proof2}

\begin{proof}[Proof of Theorem \ref{Kerov conjecture reduced}]
We can assume that $\tau(g_{1})>0$ , since if $\tau(g_{1})= 0$, then Lemma \ref{lemma:zerovalue} would imply that $\tau(Q_{\lambda}) = 0$ for any $\lambda$ . Hence in such case $\tau$ would just be the trivial specialization. 

\smallskip

To show that $\tau$ is a Plancherel specialization we need to show that $\tau(p_{k})=0$ for every $k \geq 2$. Assume that this is not true and find the smallest $k \geq 2$, such that $\tau(p_{k}) \neq 0$. Then by $\eqref{powerexponent}$ we get
\begin{align*}
\tau(g_{k}) = \frac{\left(\frac{1-t}{1-q} \tau(p_{1}) \right)^{k}}{k!} + \frac{1-t^{k}}{k\left(1-q^{k}\right)} \tau(p_{k}), \quad \text{hence} \quad \delta'  := \left \lvert \tau(g_{k}) - \frac{\tau(g_{1})^{k}}{k!} \right \rvert = \left \lvert \frac{1-t^{k}}{k\left(1-q^{k}\right)} \tau(p_{k}) \right \rvert  > 0.
\end{align*}
Then $\left \lvert \tau\left(g_{k}Q_{\mu} - g_{1}^{k}Q_{\mu}/k!\right) \right \rvert = \delta' \tau(Q_{\mu})$ for any partition $\mu$. To obtain a contradiction we will show the following lemma.
\begin{lemma}
\label{lemma:cancellation}
Under assumption \eqref{fastdecline} for any $\delta > 0$ there exists a partition $\mu$, such that 
\begin{align*}
\left \lvert \tau\left(g_{k}Q_{\mu} - \frac{g_{1}^{k}Q_{\mu}}{k!}\right) \right \rvert < \delta \tau(Q_{\mu}).
\end{align*}
\end{lemma}
So proving Lemma \ref{lemma:cancellation} would show that $\tau$ is a Plancherel specialization and complete the proof of Theorem \ref{Kerov conjecture reduced}.
\end{proof}
\begin{proof}[Proof of Lemma \ref{lemma:cancellation}]
With the help of \eqref{Pieri} we can expand both $\tau\left(g_{k}Q_{\mu}\right)$ and $\tau\left(g_{1}^{k}Q_{\mu}/k!\right)$ as weighted sums of $\tau(Q_{\lambda})$, $\mu \subset \lambda$, $|\lambda|-|\mu|=k$ with some non-negative coefficients. We will show that \eqref{fastdecline} implies a weak version of diffusivity, i.e. that for some $\mu$ most contributions to these sums come from such $\lambda$, that all boxes 
of $\lambda \backslash \mu$ are far away from each other. More precisely, define {\it distance} between boxes $(i_{1}, j_{1})$ and $(i_{2}, j_{2})$ of a partition as $|i_{1}-i_{2}|+|j_{1}-j_{2}|$. For a positive integer $d$ we will write $\mu \precd \lambda$ if $\mu \subset \lambda$ and $\lambda\backslash \mu$ contains two distinct boxes at distance $\leq d$ from each other. We will write $\mu \nprecd \lambda$ if $\mu \subset \lambda$ and the distance between any two distinct boxes of $\lambda \backslash \mu$ is greater than $d$. Then 
\begin{lemma}
\label{lemma:badsteps}
For any $\epsilon > 0$ and any positive integers $d$, $s$ there exists a partition $\mu$, such that 
\begin{align*}
\sum_{\mu \precd \lambda, \ |\lambda| - |\mu| = s} \tau(Q_{\la}) < \epsilon \tau(Q_{\mu}).
\end{align*}
\end{lemma}
We will prove Lemma \ref{lemma:badsteps} later. Then to show that the statement of Lemma \ref{lemma:cancellation} holds, we will need to  pick $\mu$ as in the statement of Lemma \ref{lemma:badsteps} and show that the coefficients of $\tau(Q_{\lambda})$ in the expansions of $\tau\left(g_{k}Q_{\mu}\right)$ and $\tau\left(g_{1}^{k}Q_{\mu}/k!\right)$ are close to each other for  $\mu \nprecd \lambda$. More precisely, we will also need the following lemma, that we will prove later.
\begin{lemma}
\label{lemma:strictpositivity}
$\tau(g_{s}) > 0$ for any $s \geq 1$.
\end{lemma}
Suppose that Lemmas \ref{lemma:strictpositivity} and \ref{lemma:badsteps} are true. Find $d>1$ such that
\begin{align*}
\left(1 + \frac{2\left(\max\{|q|, |t|\}\right)^{d-1}}{1-\max\{|q|, |t|\}}\right)^{k(k-1)} < 1 + \frac{\delta}{2\tau(g_{k})} \quad \text{and} \quad \left( 1 - \frac{2\left(\max\{|q|, |t|\}\right)^{d-1}}{1-\max\{|q|, |t|\}}\right)^{k(k-1)} > 1 -\frac{\delta}{2\tau(g_{k})}.
\end{align*}
Then for such $d$, $s=k$ and $\epsilon =\frac{\delta}{2C^{k}}$ find $\mu$, as in the statement of the Lemma \ref{lemma:badsteps}. By \eqref{Pieri} we have
\begin{multline}
\label{powermultiplication}
 \tau(g_{k}Q_{\mu})- \frac{\tau\left(g_{1}^{k}Q_{\mu}\right)}{k!}  = \sum_{\mu \prech \lambda, \ |\lambda|-|\mu|=k} \psi_{\lambda/\mu} \tau(Q_{\lambda}) - \frac{1}{k!}\sum_{\mu \subset \lambda, \ |\lambda|-|\mu|=k} \sum_{\ T \in \mathcal{ST}(\lambda/\mu)} \psi(T) \tau(Q_{\lambda}) \\
= \left(\sum_{\mu \prech \lambda, \ \mu \precd \lambda, \ |\lambda|-|\mu|=k} \psi_{\lambda/\mu} \tau(Q_{\lambda}) - \frac{1}{k!}\sum_{\mu \precd \lambda, \ |\lambda|-|\mu|=k} \sum_{\ T \in \mathcal{ST}(\lambda/\mu)} \psi(T) \tau(Q_{\lambda})\right) + \\
+ \frac{1}{k!}\left(\sum_{\mu \nprecd \lambda, \ |\lambda|-|\mu|=k} \tau(Q_{\lambda})\sum_{T \in \mathcal{ST}(\lambda/\mu)} \left(\psi_{\lambda/\mu}  - \psi(T)\right)\right)
\end{multline}
The second equality follows, since for $\mu \nprecd \lambda$ we know that $\lambda/\mu$ is both a horizontal and a vertical strip (otherwise, it would have two adjacent boxes), and there are exactly $k!$ standard tableaux of shape $\lambda/\mu$. By choice of $\mu$ and Lemma \ref{lemma:psibound}  the absolute value of the first term of the right hand side of \eqref{powermultiplication} is strictly bounded  by $C^{k}\epsilon\tau(Q_{\mu}) = \delta \tau(Q_{\mu})/2$. To finish the proof of lemma \ref{lemma:cancellation} it will be enough to show that the absolute value of the second term of the right hand side of \eqref{powermultiplication} is weakly bounded  by $\delta \tau(Q_{\mu})/2$.  Indeed, suppose $\mu, \lambda, T$ are such, that $\mu \nprecd \lambda, \ |\lambda|-|\mu|=k$, $T \in \mathcal{ST}(\lambda/\mu)$. Denote by $(i_{m}, j_{m})$ position of the entry equal to $m$ in $T$. All $i_{1}, \ldots, i_{k}$ are distinct, as well as all $j_{1}, \ldots, j_{k}$.  Then by \eqref{PsiPhi} we get
\begin{align}
\label{difference}
\frac{\psi(T)}{\psi_{\lambda/\mu}} =  \prod_{m=1 \ }^{k} \prod_{\ell: \ j_{\ell} < j_{m}} \left( b_{T_{m}}(i_{m}, j_{\ell})^{-1} \right) b_{T_{m-1}}(i_{m}, j_{\ell})
\end{align}
Each of the $k(k-1)$ terms $b(i_{m}, j_{\ell})^{\pm 1}$ of the product on the right hand side of \eqref{difference} is of the form either $\frac{1-t^{a+1}q^{b}}{1-t^{a}q^{b+1}}$ or $\frac{1-t^{a}q^{b+1}}{1-t^{a+1}q^{b}}$ for some non-negative integers $a, b$, where  the distance between boxes $(i_{m}, j_{m})$ and $(i_{\ell}, j_{\ell})$ is  $\leq a+b+2$, so $a+b \geq d-2$. Hence each term is bounded from above and from below respectively by 
\begin{align*}
1 + \frac{2\left(\max\{|q|, |t|\}\right)^{d-1}}{1-\max\{|q|, |t|\}} \quad \text{and} \quad 1 - \frac{2\left(\max\{|q|, |t|\}\right)^{d-1}}{1-\max\{|q|, |t|\}}.
\end{align*}
Hence $|\psi_{\lambda/\mu} - \psi(T)| < \frac{\delta\psi_{\lambda/\mu}}{2\tau(g_{k})}$ by our choice of $d$. Hence indeed the absolute value of the second term of the right hand side of \eqref{powermultiplication} is weakly bounded by
\begin{multline*}
\frac{1}{k!}\sum_{\mu \nprecd \lambda, \ |\lambda|-|\mu|=k} \tau(Q_{\lambda})\sum_{T \in \mathcal{ST}(\lambda/\mu)} |\psi_{\lambda/\mu}  - \psi(T)| \leq \frac{\delta}{2\tau(g_{k})}\sum_{\mu \nprecd \lambda, \ |\lambda|-|\mu|=k} \psi_{\lambda/\mu}\tau(Q_{\lambda})  \\ \leq \frac{\delta}{2\tau(g_{k})}\sum_{\mu \prech \lambda, \ |\lambda|-|\mu|=k} \psi_{\lambda/\mu}\tau(Q_{\lambda}) = \delta \tau(Q_{\mu})/2.
\end{multline*}
\begin{figure}[h]
\includegraphics[width = 
0.3\textwidth]{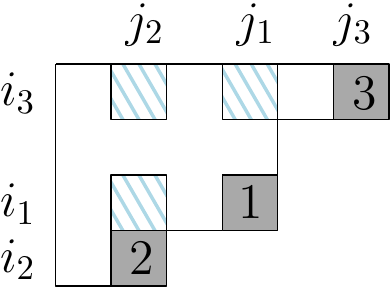}
\caption{Comparing $\psi(T)$ and $\psi_{\lambda/\mu}$. The difference in this case comes from boxes $(i_{1}, j_{2})$, $(i_{3}, j_{2})$, $(i_{3}, j_{1})$.}
\label{Attached}
\end{figure}
\end{proof}
It remains to prove Lemma \ref{lemma:strictpositivity} and Lemma \ref{lemma:badsteps}.
Repeatedly using \eqref{Pieri} we can get
\begin{align}
\label{n-level}
\tau(g_{1})^{n} = \sum_{|\lambda| = n} \sum_{T \in \mathcal{ST}(\lambda)} \psi(T) \tau(Q_{\lambda}) \quad \text{for any } n.
\end{align}
We will first show that contribution to the sum \eqref{n-level} of the terms with long first row of $\lambda$ is small.
\begin{lemma} 
\label{lemma:nolongrows}
For every $\epsilon_{1}, \epsilon_{2} > 0$ there exists such $N$, that for any $n \geq N$
\begin{align*}
\sum_{|\lambda| = n, \lambda_{1} \geq \epsilon_{1}n} \sum_{\ T \in \mathcal{ST}(\lambda)} \psi(T) \tau(Q_{\lambda}) < \epsilon_{2}^{n}.
\end{align*}
Similarly, there exists such $N'$, that for any $n \geq N'$
\begin{align*}
\sum_{|\lambda| = n, \lambda_{1}' \geq \epsilon_{1}n} \sum_{\ T \in \mathcal{ST}(\lambda)} \psi(T) \tau(Q_{\lambda}) < \epsilon_{2}^{n}.
\end{align*}
\end{lemma}

\begin{proof}[Proof of Lemma \ref{lemma:nolongrows}] We will prove only the first part of the statement, since the argument for the second part is analogous. Equation \eqref{Pieri} implies that $\tau(Q_{\lambda}) \leq \tau(g_{\lambda_{1}})\tau(Q_{\lambda^{-}})$, where $\lambda^{-}$ denotes the partition $(\lambda_{2}, \ldots, \lambda_{\ell})$. Then for $|\lambda| =n$ with the help of the lemma \ref{lemma:psibound} we get
\begin{multline*}
\sum_{\ T \in \mathcal{ST}(\lambda)} \psi(T)  \leq C^{n} |\mathcal{ST(\lambda)}| \leq C^{n}\left(\begin{array}{c} n \\ \lambda_{1} \end{array}\right)|\mathcal{ST(\lambda^{-})}| \leq \left(2C\right)^{n}|\mathcal{ST(\lambda^{-})}| \leq \left(2C^{2}\right)^{n} \sum_{\ T \in \mathcal{ST}(\lambda^{-})} \psi(T). 
\end{multline*}
Hence by \eqref{n-level}
\begin{multline*}
\sum_{|\lambda| = n, \lambda_{1} \geq \epsilon_{1}n} \sum_{\ T \in \mathcal{ST}(\lambda)} \psi(T) \tau(Q_{\lambda}) \leq \left(2C^{2}\right)^{n}\sum_{|\lambda| = n, \lambda_{1} \geq \epsilon_{1}n} \tau(g_{\lambda_{1}})\tau(Q_{\lambda^{-}})\sum_{\ T \in \mathcal{ST}(\lambda^{-})} \psi(T) \\ \leq 
\left(2C^{2}\right)^{n} \left(\sup_{p \geq \epsilon_{1}n}\tau(g_{p})\right)\sum_{m=0}^{n-1} \sum_{|\mu| = m} \sum_{\ T \in \mathcal{ST}(\mu)} \psi(T) \tau(Q_{\mu}) \leq \left(2C^{2}\right)^{n} \left(\sup_{p \geq \epsilon_{1}n}\tau(g_{p})\right) \frac{\tau(g_{1})^{n}-1}{\tau(g_{1})-1},
\end{multline*}
which for sufficiently large $n$ is less than $\epsilon_{2}^{n}$ by \eqref{fastdecline}. To prove the second part we change this argument to use $e_{k}$ instead of $g_{k}$, $P_{\lambda}$ instead of $Q_{\lambda}$, \eqref{Pieridual} instead of \eqref{Pieri}, and use the fact that $b_{\lambda} < C^{|\lambda|}$.
\end{proof}

\begin{proof}[Proof of Lemma \ref{lemma:strictpositivity}] Assume that
$\tau(g_{k})=0$. Then Lemma \ref{lemma:zerovalue} implies that $\tau(Q_{\lambda}) = 0$ for any $\lambda$ with $\lambda_{1} \geq k$. Then by \eqref{n-level}
\begin{align*}
\tau(g_{1})^{n} =  \sum_{|\lambda| = n, \ \lambda_{1} < k} \sum_{T \in \mathcal{ST}(\lambda)} \psi(T) \tau(Q_{\lambda}) \quad \text{for any } n.
\end{align*}
But each summand in the right hand side has $\lambda'_{1} \geq n/(k-1)$, hence this statement contradicts Lemma \ref{lemma:nolongrows}.
\end{proof}
\begin{remark}
In a similar manner one can prove that in fact $\tau(Q_{\lambda}) > 0$ for any partition $\lambda$, however, we don't need it at this stage of the proof of Lemma \ref{lemma:cancellation}.
\end{remark}

\begin{proof}[Proof of Lemma \ref{lemma:badsteps}]
It is enough to prove Lemma \ref{lemma:badsteps} only for $s=2$. Indeed, the general case follows, since by equation \eqref{Pieri} and Lemma \ref{lemma:psibound} we have 
\begin{align*}
\sum_{\mu \precd \lambda, \ |\lambda| - |\mu| = s} \tau(Q_{\la}) \leq C^{s-2}\tau(g_{1})^{s-2}\sum_{\mu \precd \lambda, \ |\lambda| - |\mu| = 2} \tau(Q_{\la}).
\end{align*}
Assume that for some $\epsilon > 0$ there is no such partition $\mu$ as in the statement of Lemma \ref{lemma:badsteps}. Consider a directed graph $\mathcal{D}$ with vertex set $ = \{\text{partitions } \lambda, \text{ such that } |\lambda| \text{ is even and } \tau(Q_{\lambda})>0\}$. Draw in $\mathcal{D}$ a directed edge $\mu \to \lambda$ if $\mu \precd \lambda$ and $|\lambda|-|\mu|=2$. Put on this edge weight $\tau(Q_{\lambda})/\tau(Q_{\mu})$. Define weight of any path $\emptyset \longrightarrow \lambda$ in $\mathcal{D}$ as the product of weights of all edges of this path. Then this weight is $\tau(Q_{\lambda})$. By our assumption for every vertex of $\mathcal{D}$ the sum of weights of outgoing edges is greater or equal than $\epsilon$. Hence for any $n \geq 1$ the sum of path weights over all paths   $\emptyset \longrightarrow \lambda$, over all $\lambda$ with $|\lambda| = 2n$, is greater or equal than $\epsilon^{n}$. To obtain contradiction we will show that this weighted sum in fact decays faster as $n \to \infty$.

Any path $\emptyset \longrightarrow \lambda$ in $\mathcal{D}$ corresponds to a filling of boxes of $\lambda$ with integers, which are weakly increasing along each row and down each column, such that for each $1 \leq m \leq |\lambda|/2$, entry $m$ appears exactly twice at boxes of distance $\leq d$. Denote by $\mathcal{SD}_{d}(\lambda)$ the set of all such fillings.  Then our assumption implies that
\begin{align}
\label{badstepsgrowth}
\sum_{|\lambda|=2n} \sum_{\ p: \ \emptyset \longrightarrow \lambda} \text{Weight}(p)=\sum_{|\lambda|=2n}\lvert \mathcal{SD}_{d}(\lambda) \rvert \tau(Q_{\lambda}) \geq \epsilon^{n} \quad \text{for any } n.
\end{align}
Clearly $\displaystyle \lvert \mathcal{SD}_{d}(\lambda) \rvert \leq \lvert\mathcal{ST}(\lambda)\rvert \leq C^{|\lambda|}\sum_{\ T \in \mathcal{ST}(\lambda)} \psi(T)$. By using Lemma \ref{lemma:nolongrows} we get 
\begin{align}
\label{badstepsgrowth2}
\sum_{|\lambda|=2n, \ \lambda_{1} < n}\lvert \mathcal{SD}_{d}(\lambda) \rvert \tau(Q_{\lambda}) \geq \epsilon^{n}/2 \quad \text{for all large enough } n.
\end{align}
For a partition $\lambda$ denote by $\lceil \lambda/2 \rceil$ such partition $\mu$, that $\mu_{j}' = \lceil \lambda'_{j}/2 \rceil$ for each $j \geq 1$.  Note that $|\lambda|/2 \leq \lvert\lceil \lambda/2 \rceil \rvert \leq  \left(|\lambda| + \lambda_{1}\right)/2$. To obtain a contradiction with \eqref{badstepsgrowth2} we will need the following two lemmas, which we will prove later. 
\begin{lemma}
\label{lemma:2tbaleauxbound}
There exists constant $E(d)$,  and a sequence of functions 
\begin{align*}
f_{\lambda}: \mathcal{SD}_{d}(\lambda) \to \bigcup_{\mu \subset \lceil \lambda/2 \rceil} \mathcal{ST}(\mu) \quad \text{for even } |\lambda|,
\end{align*} 
such that  $\displaystyle \sum_{|\lambda| \text{ even}}\lvert f_{\lambda}^{-1}(T) \rvert < E(d)^{|Sh(T)|}$ for any  standard tableau $T$.
\end{lemma}
Lemma \ref{lemma:2tbaleauxbound} implies the following upper bound on $\lvert \mathcal{SD}_{d}(\lambda) \rvert$: 
\begin{align}
\label{2tableauxbound}
\lvert \mathcal{SD}_{d}(\lambda) \rvert \leq \sum_{\mu \subset \lceil \lambda/2 \rceil} E(d)^{|\mu|} \lvert \mathcal{ST}(\mu) \rvert
\end{align}
\begin{lemma}
\label{lemma:comparison}
For every $\delta > 0$ there is such $N$, that $\tau(Q_{\lambda}) \leq \delta^{|\lambda|-|\mu|}\tau(Q_{\mu})$ for any partitions $\mu \subset \lambda$ with $|\lambda|-|\mu| > N$.
\end{lemma}
Suppose that both Lemmas \ref{lemma:2tbaleauxbound} and \ref{lemma:comparison} are true.  Then for any $\delta>0$ for all large enough $n$ we get 
\begin{multline*}
\sum_{|\lambda|=2n, \ \lambda_{1} < n}\lvert \mathcal{SD}_{d}(\lambda) \rvert \tau(Q_{\lambda}) \leq \sum_{|\lambda|=2n, \ \lambda_{1} < n} \sum_{\ \mu \subset \lceil \lambda/2 \rceil} E(d)^{|\mu|}\lvert \mathcal{ST}(\mu) \rvert \tau(Q_{\mu}) \delta^{|\lambda|-|\mu|}  \\
\leq E(d)^{2n}\delta^{n/2} \sum_{|\lambda|=2n, \ \lambda_{1} < n} \sum_{\ \mu \subset \lceil \lambda/2 \rceil} \lvert \mathcal{ST}(\mu) \rvert \tau(Q_{\mu})
\leq E(d)^{2n}\delta^{n/2}\frac{\tau(g_{1})^{2n}C^{2n}-1}{\tau(g_{1})C-1},
\end{multline*}
since by \eqref{n-level} and Lemma \ref{lemma:psibound}
\begin{align*}
\sum_{|\mu|=m}\lvert \mathcal{ST}(\mu) \rvert \tau(Q_{\mu}) \leq \tau(g_{1})^{m}C^{m}.
\end{align*}
But by choosing small enough $\delta$ we get a contradiction with \eqref{badstepsgrowth2}.

\end{proof}

\begin{proof}[Proof of Lemma \ref{lemma:2tbaleauxbound}]
To construct $f_{\lambda}(T)$, $T \in \mathcal{SD}_{d}(\lambda)$, we will delete some entries of $T$ and move the remaining entries to form a standard tableau of smaller size. We need to show that we can do it in such a way, so that not to lose too much information. More precisely,  $f_{\lambda}$ will be defined for  $|\lambda|=2n$ as a composition of two maps $f_{\lambda}:= f_{\lambda}^{2} \circ f_{\lambda}^{1}$.

\smallskip

\textbf{Part 1 (construction of $f_{\lambda}^{1}$).} For $1 \leq j \leq \lambda_{1}$ and $1 \leq i \leq \lfloor \lambda_{j}'/2 \rfloor$ call the set of two boxes $\left\{(2i-1, j), (2i, j)\right\}$ a {\it domino}.
$f_{\lambda}^{1}$ will take as an input any $T \in \mathcal{SD}_{d}(\lambda)$ and will produce as an output a collection $\mathcal{J}$ of $n$ boxes of $\lambda$ together with a filling of boxes of  $\mathcal{J}$ with integers $1, 2, \ldots, n$, each appearing exactly once, such that 
\begin{enumerate}
\item
Each domino contains exactly one box of $\mathcal{J}$. We will call this property {\it halfness}.

\item
If  $(i_{1}, j_{1}), (i_{2}, j_{2}) \in \mathcal{J}$, with $i_{1} \leq i_{2}$, $j_{1} \leq j_{2}$ and $(i_{1}, j_{1}) \neq (i_{2}, j_{2})$, have entries $e_1$ and $e_2$ respectively, then $e_{1} < e_{2}$. We will call this property {\it monotonicity}.
\end{enumerate}

To construct $f_{\lambda}^{1}$ consider a bipartite graph $G$, with the vertex set that is the union of the set of all dominoes and the set $\{1, 2, \ldots, n\}$.  Connect a domino $D$ and integer $k$ by an edge if at least one of the entries in $D$ is equal to $k$. Each $D$ has degree $1$ or $2$ in $G$, while each $k$ has degree $0$, $1$ or $2$. It is easy to see that each connected component of $G$ (which is not an integer - isolated vertex) is of one of the following three types:
\begin{enumerate}
\item
A single edge. It happens precisely when $D$ contains two entries equal to $k$. In such case delete one of these entries. 

\item
A simple cycle. In such case delete all the entries corresponding to even edges of the cycle (start numbering from any edge).

\item
A path of even length starting and ending in $\{1, 2, \ldots, n\}$. In such case  delete all the entries corresponding to even edges of the path (start numbering from any end).   
\end{enumerate} 
After the completion of such procedure each domino contains exactly one entry and no two dominoes contain the same integer. There still might be integers that appear in the tableau twice. For each such integer $k$ just delete one of the entries equal to $k$ that is not in any of the dominoes. The described construction involves making some choices, but we can just make any so that the rule for $f_{\lambda}^{1}$ is well-defined. Monotonicity of the image is obvious.
\begin{figure}[h]
\includegraphics[width = 
0.9\textwidth]{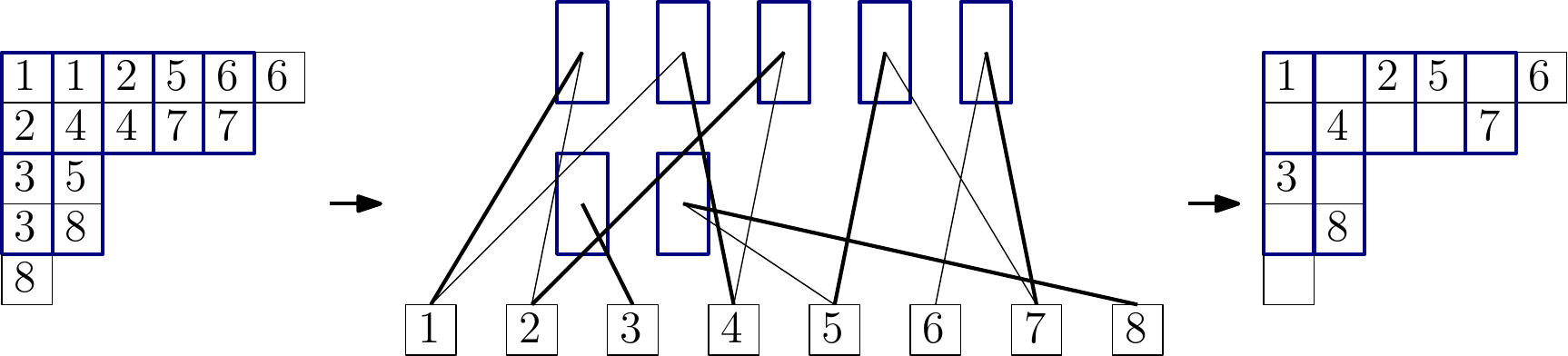}
\caption{Construction of $f^{1}_{\lambda}$. An example.}
\label{f1}
\end{figure}

\smallskip

\textbf{Part 2 (construction of $f_{\lambda}^{2}$).}
$f_{\lambda}^{2}$ will take as an input an output of $f_{\lambda}^{1}$ and will produce as an output  a standard tableau of size $n$ and shape $\subset \lceil \lambda/2 \rceil$. It will be a result of moving the remaining entries according to the  algorithm described below, so that monotonicity will still be satisfied after every step. Call an entry a {\it north} entry if it is located in an odd row, and a {\it south} entry if it is located in an even row. Order boxes lexicographically, i.e say that $(i_{1}, j_{1}) < (i_{2}, j_{2})$ if either $i_{1} < i_{2}$ or $i_{1}=i_{2}$ and $j_{1}< j_{2}$. 

\begin{enumerate}
\item
If all entries are north entries, then in each column of $\lambda$ move all entries to stack them at the top (preserving their order). If there are such entries that before this operation used to be the in the lowest box of a column of odd length and now have an empty box on the left, then in each row move all such entries as much to the left as possible (preserving their order). See Fig. \eqref{Step1}. We get a standard tableau of shape $\mu \subset \lceil \lambda/2 \rceil$ and size $|\mu|=n$, and the algorithm terminates. If there is at least one south entry, then go to step 2.
\item
Find the smallest box with a south entry. Denote this entry by $a$ and its box by $(2i, j)$. If there is no north entry with position $(2i-1, j')$ for some $j' > j$, then move $a$ to position $(2i-1, j)$ (it becomes a north entry) and go to step 1. Otherwise, go to step 3.

\item
Find the smallest $j' > j$ such that $(2i-1, j')$ has an entry. Denote this entry by $b$. If $b > a$, then move $a$ to position $(2i-1, j)$ (it becomes a north entry) and go to step 1. Otherwise, go to step 4.

\item
If $b < a$, then move $b$ to position $(2i-1, j)$ and move each entry in position $(2i-1, j'')$ for $j \leq j'' < j'$ (including $a$ itself) to position $(2i-1, j''+1)$. Note that some entries might temporarily move outside of $\lambda$. Go to step 2.
\end{enumerate}
\begin{figure}[h]
\includegraphics[width = 
0.8\textwidth]{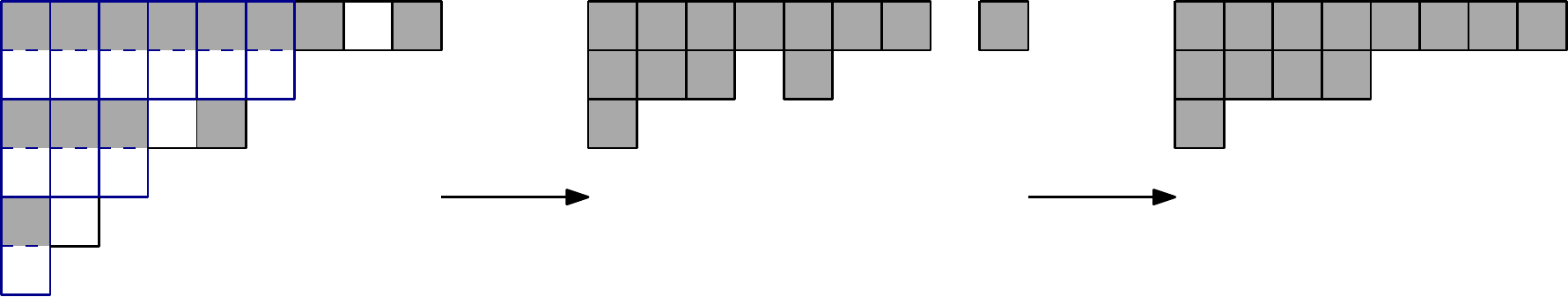}
\caption{Step 1 of the algorithm.}
\label{Step1}
\end{figure}
\begin{figure}[h]
\includegraphics[width = 
0.8\textwidth]{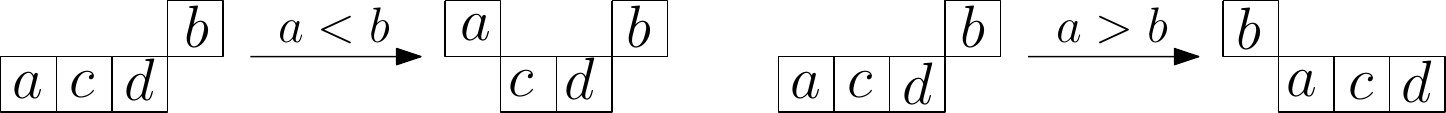}
\caption{Steps 3 and 4 of the algorithm.}
\label{Step34}
\end{figure}
Clearly, all steps of the algorithm but the final preserve halfness.  It is also straightforward to check that steps 1, 2, 3 preserve monotonicity. Let's verify it also for step 4.  After this step any entry in a box $(p, q)$ with $p \geq 2i$ and $q \geq j$ is greater than $b$, since it was $\geq a$. Any entry in a box $(p, q)$ with $p \leq 2i-2$ and $j < q \leq j'$ was $<b$, so it is less than the current entry in the box $(2i, q)$. The order in each of the rows $2i-1$ and $2i$ is preserved, the rest is straightforward.

\smallskip

After at most $ \sum_{i \geq 1} \lambda_{2i-1} \leq 2n$ moves the algorithm arrives to the stage when all south entries are eliminated. Indeed, if at some point $(2i, j)$ is the smallest box with a south entry, then $j \leq \lambda_{2i-1}$, and after the next move the smallest box with a south entry will be larger. We now fix some $T \in \mathcal{ST}(\mu)$ for a partition $\mu$ of size $n$ and show that $\sum_{|\lambda| \text{ even}}\lvert f_{\lambda}^{-1}(T) \rvert$ grows at most exponentially in $n$. For a partition $\lambda$ of size $2n$ and a collection of its boxes $\mathcal{J}$ of size $n$ that satisfies halfness denote by $N\left(\lambda, \mathcal{J}\right)$ the number of such $S \in \mathcal{SD}_{d}(\lambda)$ that $f_{\lambda}(S) = T$ and $\mathcal{J}$ is precisely the collection of nonempty boxes  of  $f_{\lambda}^{1}(S)$. Denote by $M(\mathcal{J})$ the number of such fillings of boxes of $\mathcal{J}$ with integers $1, 2, \ldots, n$, each appearing exactly once, that they satisfy monotonicity and are mapped to $T$ by $f^{2}_{\lambda}$. The described above algorithm has the property that if at some moment the set of nonempty boxes  is $\mathcal{J'}$, then there are at most two ways in which boxes of $\mathcal{J'}$ can move after the next step. If we know all the choices we made when moving boxes of $\mathcal{J}$, then we can recover from $T$ the filling of $\mathcal{J}$. Hence we have $\lvert M(\mathcal{J}) \rvert \leq 2^{2n}$. For a given box $B$ any box outside of the $(2d+1) \times (2d+1)$ square centered at $B$ has distance $> d$ from $B$. Hence given a filling of $\mathcal{J}$, there are at most $(2d+1)^{2n}$ ways to fill $\lambda \backslash \mathcal{J}$ with $1, 2, \ldots, n$ to get an element of $\mathcal{SD}_{d}(\lambda)$. So $N(\lambda, \mathcal{J}) \leq (2d+1)^{2n} M(\mathcal{J})$. We have $ \leq 3^{2n}$ choices for $\lambda$, since each $\lambda_{j}'$ is $2\mu'_{j}$ or $2\mu'_{j} \pm 1$. Given $\lambda$ there are  $\leq 2^{2n}$ choices for $\mathcal{J}$.  Hence 
\begin{align*}
\sum_{|\lambda| \text{ even}}\lvert f_{\lambda}^{-1}(T) \rvert = \sum_{\lambda} \sum_{\mathcal{J}} N(\lambda, \mathcal{J}) \leq (2d+1)^{2n} 12^{2n}.
\end{align*}
\end{proof}

\begin{proof}[Proof of Lemma \ref{lemma:comparison}]
Let $\displaystyle c: = \max \left \{\{1\} \cup \left\{ \tau(g_{r})^{1/r} \right\}_{r\geq 1} \cup \left\{ \tau(e_{r})^{1/r} \right\}_{r \geq 1} \right\}$. Find $k$, such that 
\begin{align*}
\tau(g_{r})^{1/r}, \tau(e_{r})^{1/r} < \epsilon:=\delta^{4}c^{-4}C^{-4} \text{ for any } r \geq k.
\end{align*} 
Denote by $\nu_{h}$ the partition that is the union of $\mu$ and those boxes $(i, j)$ of $\lambda\backslash\mu$ for which $\lambda_{i} - j \geq k$. To check that $\nu_{h}$ is indeed a partition note that if $\lambda_{i} - j \geq k$, then both $\lambda_{i-1} - j \geq k$ and $\lambda_{i} - (j-1) \geq k$. Similarly, denote by $\nu_{v}$ the partition that is the union of $\mu$ and those boxes $(i, j)$ of $\lambda\backslash\mu$ for which $\lambda'_{j} - i \geq k$. Let $\nu: = \nu_{h} \cup \nu_{v}$. Note that $ \displaystyle \sum_{i: \lambda_{i}-\mu_{i} \geq k} \left( \lambda_{i} -\mu_{i} \right) \geq |\nu_{h}| - |\mu|$, since each box of $\nu_{h} \backslash \ \mu$ lies in some row $i$, such that $\lambda_{i} - \mu_{i} \geq k$. 
If $|\nu_{h}| - |\mu| \geq (|\lambda|-|\mu|)/4$,  then by \eqref{Pieri} we have 
\begin{align*}
 \tau(Q_{\lambda}) \leq C^{|\lambda|-|\mu|} \left(\prod_{i \geq 1} \tau(g_{\lambda_{i}-\mu_{i}})\right) \tau(Q_{\mu}) \leq 
\epsilon^{(|\lambda|-|\mu|)/4}C^{|\lambda|-|\mu|}c^{|\lambda|-|\mu|}\tau(Q_{\mu}) =  \delta^{|\lambda|-|\mu|}\tau(Q_{\mu}). 
\end{align*}
If $|\nu_{v}| - |\mu| \geq (|\lambda|-|\mu|)/4$, then similarly by \eqref{Pieridual} we have 
\begin{align*}
 \tau(Q_{\lambda}) \leq C^{|\lambda|-|\mu|} \left(\prod_{j \geq 1} \tau(e_{\lambda'_{j}-\mu'_{j}})\right) \tau(Q_{\mu}) \leq 
\epsilon^{(|\lambda|-|\mu|)/4}C^{|\lambda|-|\mu|}c^{|\lambda|-|\mu|}\tau(Q_{\mu}) =  \delta^{|\lambda|-|\mu|}\tau(Q_{\mu}). 
\end{align*}
If both $|\nu_{h}| - |\mu| \leq (|\lambda|-|\mu|)/4$ and $|\nu_{v}| - |\mu| \leq (|\lambda|-|\mu|)/4$, then $|\nu| - |\mu| \leq (|\lambda|-|\mu|)/2$, hence $|\lambda| - |\nu| \geq (|\lambda|-|\mu|)/2$. We will now consider this case. Denote by $\mathcal{M}$ the set of such boxes $(i, j)$ that $i \leq 0$ or $j \leq 0$. Let 
\begin{align*}
S_{1} := \{(i, j) \in \lambda \backslash \nu:  (i-1, j) \in \nu \cup \mathcal{M} , \ (i, j-1) \in \nu \cup \mathcal{M}\}.
\end{align*}
Then $\nu \cup S_{1}$ is a partition and $S_{1}$ is both a vertical and a horizontal strip. Let 
\begin{align*}
S_{2} := \{(i, j) \in \lambda \backslash \left(\nu \cup S_{1}\right):  (i-1, j) \in \nu \cup S_{1} \cup \mathcal{M} , \ (i, j-1) \in \nu \cup S_{2} \cup \mathcal{M}\}.
\end{align*}
Then $\nu \cup S_{1} \cup S_{2}$ is a partition and $S_{2}$ is both a vertical and a horizontal strip. Similarly define $S_{3}, S_{4}, \ldots$. Let $\displaystyle U_{m}:= \nu \cup \bigcup_{\ell=1}^{m} S_{\ell}$. We will show that $\lambda = U_{k^{2}}$. Indeed, suppose that $s = (i, j)$ is a box in $\lambda\backslash\nu$ and consider the set $\mathcal{B}$ of boxes $(i', j')$ of $\lambda\backslash\nu$ with $i' \leq i$ and $j' \leq j$. Then $\lambda_{i} - \nu_{i} \leq k$, since otherwise we would have $(i, \nu_{i}+1) \in \nu_{h}$. Similarly $\lambda_{j}'-\nu_{j}' \leq k$, hence $|\mathcal{B}| \leq k^{2}$. For each $m$ such that $s \notin U_{m-1}$, $S_{m}$ contains at least one box from $\mathcal{B}$. Indeed, we can choose a box $b = (i'', j'') \in \mathcal{B} \backslash U_{m-1}$, such that $(i''-1, j'') \in U_{m-1} \cup \mathcal{M}$ and $(i'', j''-1) \in U_{m-1} \cup \mathcal{M}$.  Then $b \in S_{m}$.  Hence $s \in U_{k^{2}}$. Then by \eqref{Pieri} we have
\begin{multline}
\label{maxbound}
\tau(Q_{\lambda}) \leq C^{|\lambda|-|\nu|} \left(\prod_{m = 1}^{k^{2}} \tau(g_{|S_{m}|})\right) \tau(Q_{\nu})  \leq C^{|\lambda|-|\mu|} c^{|\nu|-|\mu|} \left(\prod_{m = 1}^{k^{2}} \tau(g_{|S_{m}|})\right) \tau(Q_{\mu}) \\ \leq C^{|\lambda|-|\mu|} c^{|\lambda|-|\mu|} \tau(g_{\max_{1 \leq m \leq k^{2}} \{|S_{m}|\} }) \tau(Q_{\mu}). 
\end{multline}
Here we used the fact that by \eqref{Pieri} and Lemma \ref{lemma:psibound} we have
\begin{align*}
\tau (Q_{\nu}) \leq \tau(g_{1})^{|\nu|-|\mu|}C^{|\nu|-|\mu|}\tau(Q_{\mu}) \leq (Cc)^{|\nu|-|\mu|}\tau(Q_{\mu}).
\end{align*}
But $\displaystyle \max_{1 \leq m \leq k^{2}} \{|S_{m}|\} \geq \left(|\lambda| - |\nu|\right)/k^{2} \geq \left(|\lambda|-|\mu|\right)/\left(2k^{2}\right)$, hence by \eqref{fastdecline} the right hand side of \eqref{maxbound} will be $\leq \delta^{|\lambda|-|\mu|} \tau\left(Q_{\mu}\right)$ for all sufficiently large $|\lambda|-|\mu|$.
 
\end{proof}

\section{Boundary of the Young graph with Macdonald multiplicities}
\label{sec:boundary}
In this section we prove Proposition \ref{prop:boundary}. The proof is a modification of the proof of the Vershik-Kerov ring theorem, see \cite[pp. 50-51]{BO} for the later.
\begin{proof}[Proof of Proposition \ref{prop:boundary}]
Suppose $f \in \mathcal{M}_{q, t}$. To show that $\widetilde{f}$ is a homomorphism it is enough to show that $\widetilde{f}(Q_{\mu}g_{k}) = \widetilde{f}(Q_{\mu}) \widetilde{f}(g_{k})$ for any $\mu$ and $k$. It holds for $k=1$, since $f$ is harmonic (note that $g_{1} = \frac{1-t}{1-q}h_{1}$). Fix $k \geq 2$. Suppose first that $\widetilde{f}(g_{k}) = 0$, then $\widetilde{f}(Q_{\nu}) = 0$ for any $(k) \subset \nu$. The proof is exactly the same as that of Lemma \ref{lemma:zerovalue}.
Since $Q_{\mu}g_{k}$ is is a linear combination of such $Q_{\nu}$ by \eqref{Pieri}, it follows that in this case $\widetilde{f}(Q_{\mu}g_{k}) = 0$. Suppose now that $\widetilde{f}(g_{k}) > 0$.
By Lemma \ref{lemma:psibound} we can find constants $c_{1}, c_{2} > 0$,  such that 

\begin{enumerate}
\item
$\psi_{\lambda/\mu} < c_{1}$, whenever $\lambda/\mu$ is a horizontal $k$-strip. 

\item
$\psi_{\lambda/\mu} > c_{2}$, whenever $\mu \nearrow \lambda$.
\end{enumerate}
Pick $0 < c < \min \left \{c_{2}^{k}/c_{1}, 1/\widetilde{f}(g_{k}) \right \}$. Then it follows from \eqref{Pieri} that  $Q_{\mu} \left(g_{1}^{k} - cg_{k}\right)$ for any $\mu$ can be expanded as a non-negative linear combination of $Q_{\nu}$. Let $h_{1}, h_{2}: \Lambda \to \mathbb{R}$ be functionals defined by 
\begin{align*}
h_{1}\left(Q_{\mu}\right) := \widetilde{f}\left(Q_{\mu}g_{k}\right)/\widetilde{f}\left(g_{k}\right), \quad  h_{2}\left(Q_{\mu}\right) := \widetilde{f}\left(Q_{\mu}(g_{1}^{k} - cg_{k})\right)/\widetilde{f}\left(g_{1}^{k} - cg_{k}\right).
\end{align*} 
Then $h_{1}, h_{2}$ correspond to non-negative harmonic functions on $\mathcal{Y}_{q, t}$,
\begin{multline*}
\widetilde{f} = c\widetilde{f}\left(g_{k}\right) \cdot h_{1} + \widetilde{f}\left(g_{1}^{k} - cg_{k}\right) \cdot h_{2}, \qquad c\widetilde{f}\left(g_{k}\right) > 0, \\ \widetilde{f}\left(g_{1}^{k} - cg_{k}\right) = 1 -c \widetilde{f}(g_{k}) > 0, \qquad \widetilde{f}\left(cg_{k}\right) + \widetilde{f}\left(g_{1}^{k} - cg_{k}\right) = 1.
\end{multline*}
Since $\widetilde{f}$ is extreme, it follows that $\widetilde{f} = h_{1}$, hence $\widetilde{f}(Q_{\mu}g_{k}) = \widetilde{f}(Q_{\mu}) \widetilde{f}(g_{k})$. 

It remains to show that if $\theta: \Lambda \to \mathbb{R}$ is a $(q, t)$-Macdonald-positive specialization with $\theta(g_{1}) = 1$, then it corresponds to an extreme point of $\mathcal{H}_{q, t}$. This part of the argument exactly follows \cite[p. 51]{BO}. We briefly reproduce it here for the reader's convenience. By Choquet's theorem, \cite[p. 49]{BO}, there is a probability measure $P$ supported on $\mathcal{M}_{q, t}$, such that 
\begin{align}
\label{integral}
\theta(g) = \int_{f \in \mathcal{M}_{q, t}} \widetilde{f}(g) P(df), \qquad g \in \Lambda.
\end{align}
We need to show that $P$ is in fact a delta-measure at a point of $\mathcal{M}_{q, t}$. Let $\zeta_{g}$ denote the function $f \to \widetilde{f}(g)$ viewed as a random variable defined on the probability space $(\mathcal{M}_{q, t}, P)$. By \eqref{integral} the mean of $\zeta_{g}$ is $\theta(g)$. The variance of $\zeta_{g}$ is 
\begin{align*}
 \int_{f \in \mathcal{M}_{q, t}} \widetilde{f}(g)^{2} P(df) - \theta(g)^{2} =  \int_{f \in \mathcal{M}_{q, t}} \widetilde{f}\left(g^{2} \right) P(df) - \theta\left(g^{2}\right) = 0,
\end{align*}
since $\widetilde{f}$ is multiplicative for any $f \in \mathcal{M}_{q, t}$ as we have seen in the first part of this proof. It follows that $f(g) = \theta(g)$ outside  a $P$-null subset of $\mathcal{M}_{q, t}$ (depending on $g$). So $\theta(Q_{\lambda}) = f(Q_{\lambda})$ for all partitions $\lambda$ and all $f \in \mathcal{M}_{q, t}$  outside a $P$-null subset. This implies that $P$ is a delta-measure.
 
\end{proof}

%\section{Jack's case}
%\label{sec:limiting}
%We can also prove the following Jack's limiting case of theorem \ref{Kerov conjecture}:
%\begin{theorem}
%\label{Kerov-Jack}
%For fixed $\delta>0$, a homomorphism $\theta: \Lambda \to \mathbb{R}$ takes non-negative values on all the Jack's functions $P_{\lambda}^{(1/\delta)}$ if and only if 
%\begin{align}
%\label{Kerov-condition}
%\theta(p_{1}) = \sum_{i=1}^{\infty} \alpha_{i} + \delta^{-1}\left(\sum_{j=1}^{\infty} \beta_{j} + \gamma\right), \quad \theta(p_{k}) = \sum_{i=1}^{\infty} \alpha_{i}^{k} + (-1)^{k-1}\delta^{-1}\sum_{j=1}^{\infty} \beta_{j}^{k} \text{ for all } k \geq 2,
%\end{align}
%for some $\alpha_{i}, \beta_{j}, \gamma \geq 0$, such that $\sum_{i =1}^{\infty} \alpha_{i} + \sum_{j =1}^{\infty} \beta_{j} < \infty$.
%\end{theorem}
%\begin{proof}[Proof of theorem \ref{Kerov-Jack}]
%Same as proof of theorem \ref{Kerov conjecture} with some modifications.
%\end{proof}

\end{document}